\documentclass[preprint]{elsarticle}

\usepackage{amssymb,amsmath}
\usepackage{amsthm}
\usepackage{ifthen}
\usepackage{tikz}
\usepackage{framed}
\usepackage{subfig}
\usepackage{graphicx}
\usepackage{times}

\newtheorem{theorem}{Theorem}[section]
\renewcommand{\sharp}{\#}
\theoremstyle{plain}
\newtheorem{Definition}{Definition}

\theoremstyle{break}

\newtheorem{proposition}[theorem]{Proposition}
\newtheorem{problem}[theorem]{Problem}
\newtheorem{lem}[theorem]{Lemma}

\newtheorem{examples}[theorem]{Examples}

\newtheorem{remark}[Definition]{Remark}

\newtheorem{corollary}[Definition]{Corollary}

\newcommand{\N}{{\mathbb N}}

\newcommand{\F}{{\mathbf F}}

\newcommand{\Z}{{\mathbb Z}}

\newcommand{\C}{{\mathbb C}}
\newcommand{\Q}{{\mathbb Q}}
\newcommand{\bP}{{\mathbf P}}

\newcommand{\bNP}{{\mathbf{NP}}}

\newcommand{\MSOL}{{\mathrm{MSOL}}}

\newcommand{\SOL}{{\mathrm{SOL}}}

\newcommand{\Zilber}{\mathbf{CP}}

\newif\ifskip
\skiptrue
\newif\ifrevision
\revisionfalse
\newif\ifrevised
\revisedtrue
\begin{document}
\begin{frontmatter}
\title{On the Complexity of \\ Generalized Chromatic Polynomials}

\author[ag]{A.~Goodall\fnref{fn2}}
\ead{andrew@iuuk.mff.cuni.cz}
\ead[url]{http://kam.mff.cuni.cz/~andrew}

\author[mh]{M.~Hermann\fnref{fn1}}
\ead{Miki.Hermann@lix.polytechnique.fr}
\ead[url]{http://www.lix.polytechnique.fr/Labo/Miki.Hermann}

\author[tk]{T.~Kotek\fnref{fn4}}
\ead{kotek@forsyte.at}
\ead[url]{http://forsyte.at/~kotek/}

\author[jam]{J.A.~Makowsky\corref{cor1}\fnref{fn1}}
\ead{janos@cs.technion.ac.il}
\ead[url]{http://www.cs.technion.ac.il/~janos}

\author[sn]{S.D.~Noble\fnref{fn2}}
\ead{s.noble@bbk.ac.uk}
\ead[url]{http://www.bbk.ac.uk/ems/faculty/steven-noble}
\cortext[cor1]{Corresponding author}
\fntext[fn1]{Work done in part while the authors were visiting the Simons Institute
for the Theory of Computing in Spring 2016.}
\fntext[fn2]{Supported by the Heilbronn Institute for Mathematical Research, Bristol, UK.}
\fntext[fn4]{Supported by the Austrian National Research Network S11403-
N23 (RiSE) of the Austrian Science Fund (FWF).}
\address[ag]{IUUK, MFF, Charles University, Prague, Czech Republic}
\address[mh]{LIX (CNRS, UMR 7161), \'Ecole Polytechnique,  91128 Palaiseau, France}
\address[tk]{Technische Universit\"at Wien, Institut f\"ur Informationssysteme, 1040 Wien, Austria}
\address[jam]{Department of Computer Science, Technion--IIT, 32000 Haifa, Israel}
\address[sn]{Department of Economics, Mathematics and Statistics, Birkbeck, University of London, London,
\\United Kingdom}
\begin{abstract}
\ 

\noindent
J. Makowsky and B. Zilber (2004) showed
that many variations of graph colorings, called $\Zilber$-colorings in the sequel,
give rise to graph polynomials. This is true in particular for 
harmonious colorings,
convex colorings, 
$mcc_t$-colorings, 
and rainbow colorings, and many more.
N. Linial (1986) showed
that the chromatic polynomial
$\chi(G;X)$
is $\sharp \bP$-hard to evaluate for all but three values $X=0,1,2$, where evaluation is in $\bP$.

This dichotomy includes evaluation at real or complex values, 
and has the further property that the set of points for which evaluation is in {\bf P} is finite.
We investigate how the complexity of evaluating univariate graph polynomials 
that arise from $\Zilber$-colorings varies for different evaluation points. 
We show that for some $\Zilber$-colorings (harmonious, convex) the complexity of evaluation follows a similar pattern
to the chromatic polynomial. However, in other cases (proper edge colorings, $mcc_t$-colorings, $H$-free colorings) 
we could only obtain a dichotomy for evaluations at non-negative integer points.
We also discuss some $\Zilber$-colorings where we only have very partial results.
\end{abstract}

\end{frontmatter}
\tableofcontents
\newpage
\section{Introduction}
\label{se:intro}
By a classical result of R. Ladner, and its generalization by K. Ambos-Spies, 
\cite{ar:ladner75,ar:Ambos-Spies87}, 
there are infinitely many degrees (via polynomial time reducibility) between $\bP$ and $\bNP$,
and between $\bP$ and $\sharp\bP$, provided $\bP \neq \bNP$.
In contrast to this, the complexity of evaluating partition 
functions or counting graph homomorphisms
satisfies a dichotomy theorem: either 
evaluation is in $\bP$ or it is $\sharp\bP$-complete,
\cite{ar:DyerGreenhill2000,ar:BulatovGrohe2005,ar:CaiChenLu2013}.
For the definition of the complexity class $\#\bP$,
see \cite{bk:GJ} or \cite{bk:papadimitriou94}.

In accordance with the literature in graph theory
a finite graph $G =(V(G), E(G))$ with $n(G) = |V(G)|$ and $e(G)= |E(G)|$  has {\em order} $n(G)$ and {\em size} $e(G)$.
Otherwise, the {\em size of a finite set} is its cardinality.

In this paper we study the {\em complexity  of the evaluation} of generalized univariate chromatic polynomials,
as introduced in  \cite{ar:MakowskyZilber2006} and further studied in \cite{ar:KotekMakowskyZilber08,ar:KotekMakowskyZilber11}.
They will be called in the sequel $\Zilber$-colorings 
(for {\bf C}ounting {\bf P}olynomials).  
Among these we find:
\begin{examples}
\label{gp-list}
\begin{enumerate}[(i)]
\item
Trivial (unrestricted) vertex colorings using at most $k$ colors are just functions $V(G) \rightarrow [k]$.
We denote by $\chi_{trivial}(G; k)$ the number of trivial colorings of $G$, hence
$\chi_{trivial}(G; k) = k^{|V(G)|} \in \Z[k]$.
\item
Proper vertex colorings using at most $k$ colors, where two neighboring vertices receive different colors, 
are counted by $\chi(G;k)$, the classical chromatic polynomial .
\item
Proper edge colorings using at most $k$ colors, where two edges with a common vertex receive different colors, 
are counted by $\chi_{edge}(G;k)$, the edge chromatic polynomial. 
We note that they are exactly the proper vertex colorings of the
line graph $L(G)$ of $G$.
\item
Convex colorings using at most $k$ colors are vertex colorings, which are not necessarily proper, but where each color class
induces a connected subgraph. 
They are counted by $\chi_{convex}(G;k)$. 
Convex colorings are first introduced in \cite{ar:MS07}.
\item
Harmonious colorings using at most $k$ colors are proper vertex colorings such that no two edges 
have end-vertices receiving the same pair of colors.
They were introduced in \cite{ar:HopcroftKrishnamoorthy83,ar:EdwardsMcDiarmid95,ar:Edwards97}.
We denote the number of harmonious colorings using at most $k$ colors by 
$\chi_{harm}(G;k)$. The graph parameter 
$\chi_{harm}(G;X)$ is a polynomial in $k$ by \cite{ar:MakowskyZilber2006,ar:KotekMakowskyZilber11}
which was further studied more recently in \cite{ar:DrgasGibek2017}.
\item
For a fixed connected graph $H$,
$DU(H)$-colorings  are vertex colorings, where each color class induces a disjoint collection of copies of $H$.
The graph parameter counting the number of $DU(H)$-colorings with at most $k$ colors
is a polynomial in $k$, and the corresponding graph polynomial is denoted by $\chi_{DU(H)}(G;k)$.
\item
For a fixed $t \in \N^+$, an
$mcc_t$-coloring using at most $k$ colors is a vertex coloring, where the connected components of the subgraphs induced 
by each color class have at most $t$ vertices.
They were previously studied in \cite{ar:ADV03} and \cite{ar:LMST07a}.
The graph parameter  $\chi_{mcc_t}(G;k)$
counting the number of $mcc_t$-colorings with at most $k$ colors
is also a polynomial in $k$ but not in $t$.
\item
For a fixed graph $H$, an $H$-free coloring using at most $k$ colors is a vertex coloring in which every color class 
induces an $H$-free graph.
For $H=K_2$ these are the proper  vertex colorings.
The graph parameter  $\chi_{H-free}(G;k)$
counting the number of $H$free colorings with at most $k$ colors
is also a polynomial in $k$.
\end{enumerate}
\end{examples}
More examples are presented in Section \ref{se:chromatic},
where we also discuss a general theorem
which allows us to find infinitely many generalized chromatic polynomials,
and in Section \ref{se:varia}. 

\subsection{The complexity spectrum}
Let $\F$ be a fixed field that contains $\Q$, the rational numbers, and in which the arithmetic operations are polynomial
time computable.
For our discussion we use the unit-cost\footnote{
If instead we use the binary cost model for computations in, say, $\Q$, the main results still hold, but have to
be formulated more carefully, as $a \in \Q$ could be very large, and the notion of uniformity would be affected.
} model for the field computations in $\F$.

Given a graph polynomial $P(G;X) \in \F[X]$ and an element $a \in \F$, we view $P_a(G)= P(G;a)$ as a graph parameter.
We will look at the complexity of the problem of evaluating $P(G;a)$ for  a fixed $a \in \F$ and at the problem
of computing all the coefficients of $P(G;X)$.
\begin{framed}
\begin{trivlist}\item[]
\textbf{Problem:} $P(G;a)$\\
\textbf{Input:} Graph $G$.\\
\textbf{Output:} The value of $P(G;X)$ for $X=a$.
\end{trivlist}
\end{framed}
\begin{framed}
\begin{trivlist}\item[]
\textbf{Problem:} $P(G;X)$\\
\textbf{Input:} Graph $G$.\\
\textbf{Output:} All the coefficients of $P(G;X)$ as a vector over $\F$.
\end{trivlist}
\end{framed}
We denote by $T_{P_a}(n)$ the time needed to compute $P(G;a)$ on graphs with $n$ vertices in the Turing model of computation.
Similarly $T_{P_X}(n)$ denotes the time needed to compute all the coefficients of $P(G;X)$.
Clearly, for every $a \in \F$ the problem  $P(G;a)$ is reducible to computing the coefficients of $P(G;X)$.
The converse is not true in general, but we shall see cases where for certain $a_0 \in \F$ computing the coefficients
of $P(G;X)$ is reducible to $P(G;a_0)$. 
When we speak informally of the {\em complexity spectrum of $P(G;X)$} 
we have in mind the variability of $T_{P_a}(n)$ where $a \in \F$,
without giving the term a precise definition.
For a graph polynomial $P(G;X)$,
we are interested in describing the complexity of $P_a(G)$ for all $a \in \F$.
A more modest task would be to describe it only for $a \in \N$. In the case of $a \in \N$ we speak of
the {\em discrete complexity spectrum}, in the case of $a \in \F$ we speak 
of the {\em full complexity spectrum}, 
if the context requires it. 

We define
$$
\mathrm{EASY}(P)= \{ a \in \F: \mbox{ there exists } d \in \N \mbox{ with } T_{P_a}(n) \leq n^{d} \mbox{ for all } n \geq 2 \}
$$
Analogously, we define
$$\mathrm{\sharp PHARD}(P)= \{ a \in \F: P_a(G) \mbox{  is } \sharp\bP\mbox{-hard}  \}.$$

\subsection{Easy computation of the polynomial versus its easy evaluation}
In the definition of $\mathrm{EASY}(P)$, we require that the evaluation at 
$a \in \F$ takes polynomial time,
but the exponent $d$ may depend on $a$, which is to say 
the polynomial bound for evaluating $P(G;a)$ is {\em non-uniform}.

If we could compute all of the coefficients of $P(G;X)$ in polynomial time,
we could also evaluate $P(G;a)$ for every $a$ in polynomial time $O(n^d)$ where
$d$ is independent of $a$, and $P(G;a)$ can be evaluated in polynomial time uniformly.

However, $\mathrm{EASY}(P)$ is infinite with a non-uniform polynomial bound, 
then it does not follow that there exists $d$ such that for every $G$ the coefficients of $P(G;X)$ 
maybe computed in time $O(n^d)$.

\begin{proposition}
\begin{enumerate}[(i)]
\item
If
the coefficients of 
the polynomial $P(G;X)$ can be computed in polynomial time from $G$ alone, 
then $\mathrm{EASY}(P) = \F$.
\item
If, in addition to (i), there is $a \in \F$
such that $a \in \mathrm{\sharp PHARD}(P)$, then $\bP = \sharp\bP$.
\end{enumerate}
\end{proposition}
\begin{proof}
As all the coefficients of $P(G;X)$ can be computed together in polynomial time,
the degree $\delta= \delta(G)$ of $P(G;X)$ is polynomial in the number of vertices of $G$, 
and so is the size of the coefficients.

Evaluating such a polynomial can be done in polynomial time in the order of $G$.
Therefore if $a \in \mathrm{EASY}(P) \cap \mathrm{\sharp PHARD}(P)$ we have $\bP = \sharp\bP$.
\end{proof}

The characteristic polynomial $p_A(G;X)$ of a graph $G$ is the characteristic polynomial 
of its adjacency matrix.
A variant of this, $p_L(G;X)$, is obtained by replacing the adjacency 
matrix by the Laplacian of $G$.
Both are obtained by computing a determinant, therefore evaluating both $p_A(G;X)$ and $p_L(G;X)$
can be done in time $O(n^3)$, irrespective of the evaluation 
point $X=a \in \F$. Hence $\mathrm{EASY}(p_A(G;X))= \F$ ($\mathrm{EASY}(p_L(G;X))= \F$) uniformly.

In Section \ref{se:harmonious} we shall see an example where $\mathrm{EASY}(P) = \N$ non-uniformly.

\subsection{Linial's Trick}
In the case of the chromatic polynomial $\chi(G;X)$, N. Linial set the paradigm in the
following theorem:
\begin{theorem}[\cite{ar:Linial86}]
\label{th:linial}
$\mathrm{\sharp PHARD}(\chi)= \F - \{0,1,2\}$ and $\mathrm{EASY}(\chi) = \{0,1,2\}$.
\end{theorem}

To show this N. Linial
observed the following:
\begin{lem}[Linial's Trick]
\label{LTrick}
Let $G_1 \bowtie G_2$ be the join of the graphs $G_1$ and $G_2$, obtained from the disjoint union $G_1 \sqcup G_2$
by connecting all the vertices of $G_1$ with all the vertices of $G_2$.
\begin{enumerate}[(i)]
\item
A function  $f: V(G \bowtie K_1) \rightarrow [k]$ is a proper coloring  with $k$ colors 
if and only if
there is a function $g: V(G) \rightarrow \{1, \ldots, i-1, i+1, \ldots, k\}$
which is a proper coloring of $G$ with $k-1$ colors
and $f|_{V(G)}=g$, i.e., $g$ is the restriction of $f$ to $V(G)$, and $f(u)=i$.
\item
$ \chi(G \bowtie K_1;k) =  k \cdot \chi(G; k-1) $
\item
$ \chi(G \bowtie K_n;k) =  k_{(n)} \cdot \chi(G; k-n) $
where 
$k_{(n)}$ is the falling factorial:\\
$k_{(n)} = k (k-1)(k-2)\ldots (k- (n-1)) = \frac{k!}{(n-k)!}$.
\end{enumerate}
\end{lem}

This allows one for $a\not\in\N$ to evaluate $\chi(G,a-n)$ by computing $\chi(G\bowtie K_n, a)$. 
It also shows that evaluating $\chi(G; k)$ is reducible to evaluating $\chi(G; k+1)$ for $k \geq 3$.
\begin{proof}[Proof of Theorem \ref{th:linial}]
First one proves that $\chi(G;3)$ is $\sharp\bP$-complete directly, as in \cite{ar:Linial86}.
For the cases $a \in \N - \{0,1,2\}$ one uses (ii) of Linial's trick iteratively to 
reduce the computation of $\chi(G; 3)$ to the computation of $\chi(G; a)$.
For the cases $a \in \F - \N$  one uses (iii) of Linial's trick to compute 
$\chi(G;a-n)$ for $n=0,1,\dots, n(G)$. 
By Lagrange interpolation the polynomial $\chi(G;X)$ is thereby determined since the degree of $\chi(G;X)$ is $n(G)$. 
In particular computing $\chi(G;3)$ is polynomial time reducible to computing $\chi(G;a)$. 
Hence, Theorem \ref{th:linial} follows.
\end{proof}

Similarly, for the generating matching polynomial 
$$gm(G;X) = \sum_{M \subseteq E(G)} X^{|M|},$$
we have, see
\cite{ar:AverbouchMakowsky07}\footnote{
It appears that this was known as folklore, but we could not find a suitable reference.}.

\begin{proposition}
\label{prop:matching}
$\mathrm{\sharp PHARD}(gm) = \F -\{0\}$ and $\mathrm{EASY}(gm)= \{0\}$, 
\end{proposition}
\small
\begin{center}
\begin{table}[htb]
\begin{tabular}{|l | l| l| l | l|}
\hline
&&&& \\
G-polynomial & $E=\mathrm{EASY}(P)$ & $\mathrm{\sharp PHARD}(P)$ & $\mathrm{OTHER}$ & Reference\\
&&&&\\
\hline
\vspace{0.05cm}
$\chi_{trivial}(G;X)$ & $E_{trivial}= \F$, u& $\emptyset$ & $\emptyset$ & trivial\\
\vspace{0.05cm}
$p_A(G;X)$ & $E_{char}= \F$, u& $\emptyset$ & $\emptyset$ & folklore\\
\vspace{0.05cm}
$gm(G;X)$ & $E_{match}= \{0\}$  & $\F- E_{match}$ & $\emptyset$ &  folklore \\
\vspace{0.05cm}
$\chi(G;X)$ & $E_{chrom}= \{0,1,2\}$ & $\F - E_{chrom}$ & $\emptyset$ & Theorem \ref{th:linial}\\
\vspace{0.05cm}
$\chi_{harm}(G;X)$ & $E_{harm}= \N$, nu& $\F - E_{harm}$ & $\emptyset$ & Theorem \ref{th:harmonious}\\
\vspace{0.05cm}
$\chi_{convex}(G;X)$ & $E_{convex}= \{0,1\}$ & $\F - E_{convex}$ & $\emptyset$ & Theorem \ref{th:convex}\\
\vspace{0.05cm}
$\chi_{DU(K_{\alpha})}(G;X)$ & $E_{DU(K_{\alpha})}= \{0,1\}$ & $\F - E_{DU(K_{\alpha})}$ & $\emptyset$ & Theorem \ref{th:cliquealpha}\\
$\alpha \geq 2$ &  & &  & \\
\hline
\end{tabular}
\caption{Full complexity spectra, u=uniformly, nu=non-uniformly}
\label{table-1}
\end{table}
\end{center}

\begin{center}
\begin{table}[htb]
\begin{tabular}{|l | l| l| l | l|}
\hline
&&&& \\
G-polynomial & $E=\mathrm{EASY}(P)$ & $\mathrm{\sharp PHARD}(P)$ & $\mathrm{OTHER}$ & Reference\\
&&&&\\
\hline
\vspace{0.05cm}
$\chi_{edge}(G;X)$ & $E_{edge}= \{0,1\}$ & $\N - E_{edge}$ & $\emptyset$ & Theorem \ref{th:edgecoloring}\\
\vspace{0.05cm}
$\chi_{mcc_t}(G;X)$ & $E_{mcc_t}= \{0,1\}$ & $\N - E_{mcc_2}$ & $\emptyset$ & Theorem \ref{th:mcc-new}\\
$t \geq 2, k \geq 2$ &  & &  & \\
\vspace{0.05cm}
$\chi_{H-free}(G;X)$ & $E_{H-free}= \{0,1\}$ & $\N - \{0,1,2\}$ &  $\{2\}$ (+) & Theorem \ref{th:h-free}\\
\hline
\end{tabular}
\caption{Discrete complexity spectra only, $H$ of size $2$, (+) only $\bNP$-hard is known}
\label{table-2}
\end{table}
\end{center}

\ifskip
\else
\begin{center}
\begin{tabular}[pos]{|l | l| l| l |l|}
\hline
&&&&\\
$P$ & $\mathrm{EASY}(P)$ & $\mathrm{NPHARD}(P)$ & $\mathrm{\sharp PHARD}(P)$ & $\mathrm{OTHER}$\\
\hline
\hline
&&&&\\
$\chi(G;X)$ & $E_1= \{0,1,2\}$ & $\F - E_1$ & $\F - E_1$ & $\emptyset$\\
\hline
\end{tabular}
\end{center}
\normalsize
\fi 
\normalsize

The purpose of this paper is to study the complexity spectrum of 
generalized univariate chromatic polynomials
arising from $\Zilber$-colorings.

In the examples we study, the complexity spectrum is easily described with the two sets
$\mathrm{EASY}(P)$
and
$\mathrm{\sharp PHARD}(P)$.
Our results on full complexity spectra are summarized in Table \ref{table-1}.
Cases where only the discrete spectrum is understood are given in Table \ref{table-2}.

To get a complete description of the full complexity spectrum of $P(G;a)$, one needs two ingredients:
\begin{enumerate}[(i)]
\item
Enough points $a \in \N$ for which the complexity of $P_a(G)= P(G;a)$ is known, and
\item
some form of reducibility between $P_a(G)$ and $P_b(G)$ for the remaining values $a,b \in \F$.
\end{enumerate}
From the literature we often, but not always, can get enough information for (i).
We give here new results for (i), namely Theorems \ref{th:goodall-noble}, \ref{th:DU-hard} and \ref{th:mcc_2}.
For (ii) we try to adapt Linial's Trick, which in some cases is more or less straightforward, while in other cases
requires finding a new gadget as in Theorem \ref{th:harmonious}.
A precise description of what is needed for (ii) is given in \cite{ar:BlaeserDellMakowsky10}, which also
covers the case for multivariate graph polynomials.

\subsection{The Difficult Point Dichotomy} 
We say that a univariate graph polynomial has the {\em Difficult Point Dichotomy}
if 
\begin{enumerate}[(i)]
\item
for every $a \in \F$ either $a \in \mathrm{EASY}(P)$ or $a \in \mathrm{\sharp PHARD}(P)$, and
\item
$\mathrm{EASY}(P) = \F$ or $\mathrm{EASY}(P) \subseteq \N$.
\end{enumerate}

In \cite{ar:MakowskyZoo,pr:MakowskyKotekRavve2013} it is conjectured that, for every univariate graph polynomial 
$P(G;X)$ definable in Second Order Logic $\SOL$, the set $\mathrm{EASY}(P)$ is either finite or $\mathrm{EASY}(P) =\F$
and that $\mathrm{EASY}(P) \cup \mathrm{\sharp PHARD}(P) =\F$.
In \cite{ar:MakowskyZoo} the same was also conjectured for univariate graph polynomials definable 
in Monadic Second Order Logic $\MSOL$.
The example of  $\chi_{harm}(G;X)$ is $\SOL$-definable and therefore disproves the conjecture for
$\SOL$-definable graph polynomials. However, it was shown in \cite{ar:KotekMakowsky-LMCS2014} that $\chi_{harm}(G;X)$ is not
$\MSOL$-definable. 
Rather than conjecturing frivolously, we state some problems.

\begin{problem}
\label{prob-1}
Which univariate graph polynomials $P(G;X)$ satisfy the Difficult Point Dichotomy? 

In particular,
\begin{enumerate}[(i)]
\item
Is it true for every $\MSOL$-definable univariate graph polynomial $P(G;X)$?
\item
Can one find a criterion which applies to an infinite family of univariate graph polynomials $P(G;X)$
which are not partition functions, or which do not count homomorphisms,
and which implies the Difficult Point Dichotomy.
\end{enumerate}
\end{problem}

\ifskip
\else
\begin{problem}
Prove or disprove:
Every $\SOL$-definable univariate graph polynomial $P(G;X)$ 
with $\mathrm{EASY}(P) \neq \F$ satisfies the following:
\begin{enumerate}[(i)]
\item
$\mathrm{EASY}(P) = \N$  or $\mathrm{EASY}(P)$ is a finite subset of $\N$, and
\item
$\mathrm{EASY}(P) \cup \mathrm{\sharp PHARD}(P) =\F$.
\end{enumerate}
\end{problem}

\begin{problem}
\abel{prob-2}
Prove or disprove:
Every $\MSOL$-definable univariate graph polynomial $P(G;X)$ 
with $\mathrm{EASY}(P) \neq \F$ satisfies the following:
\begin{enumerate}[(i)]
\item
$\mathrm{EASY}(P)$ is a finite subset of $\N$, and
\item
$\mathrm{EASY}(P) \cup \mathrm{\sharp PHARD}(P) =\F$.
\end{enumerate}
\end{problem}
\begin{framed}
JAM: I am not sure about the formulation of the problems.
\end{framed}
\fi 

\subsection*{Outline of the paper}
In Section \ref{se:chromatic} we give a simplified proof of the result from \cite{ar:MakowskyZilber2006}
that not only counting proper graph colorings,
but counting many other graph colorings, give rise to  infinitely many generalized chromatic graph polynomials. 
We give many explicit examples, and show that there are infinitely many such graph polynomials which are
mutually semantically incomparable, cf. \cite{ar:MakowskyRavveBlanchard2014,ar:KotekMakowskyRavve2017}. 
In Sections 
\ref{se:explicit} 
and
\ref{se:explicit-1} 
we analyze in detail the graph polynomials from Table \ref{table-1} and \ref{table-2}.
In Section \ref{se:GN} we give the proof that counting convex colorings with $2$ colors is $\sharp\bP$-complete
(Theorem \ref{th:goodall-noble}).
In Section \ref{se:varia} we discuss graph polynomials for which we have only partial results.
Finally, in Section \ref{se:conclu} we summarize our conclusions and list some open problems.


\section{One, two, many chromatic polynomials}
\label{se:chromatic}
Let $G=(V(G),E(G))$ be a finite graph and $k \in \N^+$ a positive integer. 
We denote the set $\{1, \ldots , k\}$ by $[k]$.
Unless otherwise stated all graphs are simple, i.e., loop-free and without multiple edges.
A {\em vertex (edge) coloring $f$} with $k$  colors is a function
$f: V(G) \rightarrow [k]$
($f: E(G) \rightarrow [k]$).
The coloring $f$ is {\em proper} if no two vertices (edges) with a common edge (vertex) have the same color.

Let 
$\chi(G;k)$
($\chi_{edge}(G;k)$)
denote the number of {\em proper} vertex (edge) colorings of $G$ with $k$ colors.
In 1912 G. Birkhoff \cite{ar:Birkhoff1912} noticed that $\chi(G;k)$
and $\chi_{edge}(G;k)$ are polynomials in $\Z[k]$, and therefore can be extended to
polynomials in $\C[X]$, 
denoted,
by abuse of notation, 
by $\chi(G;X)$ and $\chi_{edge}(G;X)$.
Birkhoff's proof for $\chi(G;X)$ was based on a recurrence relation involving deletion and contraction of edges,
which was generalized and led, in its most general form, to the Tutte polynomial.
For $\chi_{edge}(G;X)$ one simply observes that 
\begin{gather}
\label{linegraph}
\chi(L(G);X) = \chi_{edge}(G;X), 
\tag{*}
\end{gather}
where $L(G)$ is the line graph of $G$.
Although proper edge colorings have been studied in the literature, the polynomial $\chi_{edge}(G;X)$ has not received
wide attention, probably because of (\ref{linegraph}).

\subsection{Many chromatic polynomials}
Let $\mathcal{G}$ denote the class of all finite graphs.
We introduce our concepts for vertex colorings, but they can be straightforwardly extended to
edge colorings.

Two vertex colorings $f_1, f_2:V(G) \rightarrow [k]$ are isomorphic 
if there
is an automorphism $\alpha: V(G) \rightarrow V(G)$ of $G$ and a permutation $\pi: [k] \rightarrow [k]$ such that
for all $v \in V(G)$
$$
\pi(f_1(v)) = f_2(\alpha(v)).
$$
Let 
$\mathrm{COL}=  \bigcup_{G \in \mathcal{G}} \bigcup_{k \in \N^+} [k]^{V(G)}$.
A {\em coloring property $\Phi$} is a subset of $\mathrm{COL}$ that is closed under isomorphisms
of colorings.

For a fixed coloring property $\Phi$, a graph $G \in \mathcal{G}$, $k \in \N^+$ and $I \subseteq [k]$,
let 
$$
\chi_{\Phi}(G;k) = |\{f \in \Phi: f: V(G) \rightarrow [k] \}|
$$
and let
$$
c_G^{\Phi}(I,k)
= |\{f \in \Phi: f: V(G) \rightarrow [k] \mbox{   with   } f(V(G))=I \}| 
$$
be the number of colorings $f \in \Phi$ of $G$ which use exactly the colors in $I$.

\ifrevised
We say that $\Phi$ is a {\em $\Zilber$-Property} if the following two conditions are satisfied:
\begin{description}
\item[(A)]
For all $k \in \N^+$ and $I,J \subseteq [k]$  with $|I| =|J| =i$ 
we have
$$c^{\Phi}_G(I,k) = c^{\Phi}_G(J,k).$$
\item[(B)]
For all $k, k' \in \N^+$ with $I \subseteq [k]\cap [k']$ we have
$c^{\Phi}_G(I,k) = c^{\Phi}_G(I,k')$.
\end{description}
If (A) holds we let
$c^{\Phi}_G(i,k)$ 
denote the common value of
$c^{\Phi}_G(I,k)$, where $|I|=i$, and if both (A) and (B) hold, we let
$c^{\Phi}_G(i)$
denote the common value of
$c^{\Phi}_G(i,k)$, for $k \geq i$. 
\fi 

We now compute
$\chi_{\Phi}(G;k)$ using (A) and (B):
$$
\chi_{\Phi}(G;k) = \sum_{I \subseteq [k]} c^{\Phi}_G(I,k) =  \sum_i c^{\Phi}_G(i) {k \choose i}.
$$
This establishes the following result, first shown in [MZ06]:
\begin{theorem}[\cite{ar:MakowskyZilber2006}]
If $\Phi$ is a $\Zilber$-property, the counting function $\chi_{\Phi}(G;k)$ is a polynomial in $\Z[k]$.
\end{theorem}

\begin{examples}
\begin{enumerate}[(i)]
\item
In the case of the chromatic polynomial, both (A) and (B) are satisfied. Hence we get a new proof
of Birkhoff's Theorem.
\item
Let
$\hat{\chi}_{\Phi}(G;k)= c^{\Phi}_G(k)$
be the graph parameter that counts the number of colorings in $\Phi$ of $G$ which use exactly $k$ colors.

Note that the function $\hat{\chi}_{\Phi}(G;k)$ need not be a polynomial in $k$
when $\chi_{\Phi}(G;k)$ is a polynomial in $k$.
\item
Let $\Phi_1$ be the coloring property which says $f$ is a proper coloring  with $k$ colors where all the $k$ 
colors are used. Here (B) is violated, and indeed, $\chi_{\Phi_1}(G;k)$ is not a polynomial.
\item
Let $\Phi_2$  be the coloring property which says $f$ is a proper coloring  with $k$ colors such that
$f(v) =i+1$ if and only if the degree of $v$ is $i$. Here (A) is violated, but (B) is still true, 
and $\chi_{\Phi_2}(G;k)$ is still a polynomial.
\item
All the examples (i)-(vii) of Section \ref{se:intro} listed in Examples \ref{gp-list} satisfy (A) and (B).
Hence they are polynomials in $k$.
\end{enumerate}
\end{examples}

\subsection{$\mathcal{P}$-colorings and variations}
\label{subse:P-colorings}

Two graph polynomials may be compared via their {\em distinctive power}.
Two graphs $G_1$ and $G_2$  are  {\em similar}
if they have
the same number of vertices, edges and connected components.
A graph polynomial $Q(G; X)$ is {\em less distinctive than} $P(G;Y)$, written $Q \preceq P$, 
if for every two similar graphs $G_1$ and $G_2$ 
\begin{gather}
P(G_1;X) =P(G_2;X) \mbox{ implies }Q(G_1;Y) =Q(G_2;Y).
\notag
\end{gather}
We also say that $P(G;X)$ {\em determines $Q(G;X)$} if $Q \preceq P$.
Two graph polynomials $P(G;X)$ and $Q(G;Y)$  are 
equivalent in distinctive power (d.p.-equivalent) 
if for every two similar graphs $G_1$ and $G_2$ 
\begin{gather}
P(G_1;X) =P(G_2;X) \mbox{ iff } Q(G_1;Y) =Q(G_2;Y).
\notag
\end{gather}
Here we show how to obtain infinitely many graph polynomials that 
are mutually incomparable in distinctive power.

Let $\mathcal{P}$ be any graph property (a class of finite graphs closed under graph isomorphism).

A function $f: V(G) \rightarrow [k]$ is a {\em $\mathcal{P}$-coloring}
if for every $i \in [k]$ the set $f^{-1}(i)$ induces a graph $G[f^{-1}(i)] \in \mathcal{P}$.
Clearly, this is a $\Zilber$-coloring for any graph property $\mathcal{P}$. Hence
$\chi_{\mathcal{P}}(G;k)$, the number of $\mathcal{P}$-colorings of $G$ with at most $k$ colors,
is a polynomial in $k$.

\ifrevised
\begin{theorem}[\cite{ar:KotekMakowskyRavve2017}] 
\label{th:dp}
There are infinitely many graph polynomials of the form
$\chi_{\mathcal{P}}(G;k)$
with mutually incomparable distinctive power.
\end{theorem}
\fi 

We can generalize this further.
Let $\mathcal{P}_1$ and $\mathcal{P}_2$ be two graph properties.
The $\mathcal{P}_1$-colorings such that the union of any two color classes induces a graph in $\mathcal{P}_2$
form also a $\Zilber$-property.
Let 
\ifrevised
\begin{gather}
\chi_{\mathcal{P}_1, \mathcal{P}_2}(G;k) = \notag \\
|\{f:V(G)\to [k]: \forall_{i\in [k]} G[f^{-1}(\{i\})]\in \mathcal P_1,
\forall_{i,j\in [k], i \neq j} G[f^{-1}(\{i,j\})]\in \mathcal P_2\}|
\notag
\end{gather}
\fi 
denote the number of such colorings.
Then, 
for $k\in\mathbb N$ and a graph $G$, 
the graph invariant $\chi_{\mathcal{P}_1, \mathcal{P}_2}(G;k)$
is a polynomial in $k$.

\begin{problem}
\label{prob-3}
\label{prob:F1F2} 
For which graph properties
$\mathcal{P}_1, \mathcal{P}_2$  can we describe the complexity of
$\chi_{\mathcal{P}_1, \mathcal{P}_2}(G;k)$? In particular, for which graph properties does 
the Difficult Point Dichotomy hold? 
\end{problem}

\begin{remark}
In Problem
\ref{prob:F1F2} 
it may be reasonable to impose some complexity restrictions on  $\mathcal{P}_1$ and $\mathcal{P}_2$,
e.g., we might require them to be in $\bNP$.
\end{remark}

Let $\mathcal{AH}$ be
any additive induced hereditary property 
(closed under taking induced subgraphs and disjoint unions).
$\mathcal{AH}$-colorings have been studied in \cite{ar:Brown1996,ar:Farrugia04},
in which the following is shown:
\begin{theorem}
\begin{enumerate}[(i)]
\item
(\cite{ar:Brown1996})
There are uncountably many induced hereditary properties $\mathcal{AH}$ of graphs.
Therefore, $\chi_{\mathcal{AH}}(G; k)$ may not be computable.
\item
(\cite{ar:Farrugia04})
$\chi_{\mathcal{AH}}(G; k)$ is $\bNP$-hard, unless $\mathcal{AH}$ is the class of empty (=edgeless) graphs.
\end{enumerate}
\end{theorem}

Table~\ref{table-zilber} unifies the $\Zilber$-colorings 
considered in Table \ref{table-1} 
as $\mathcal{P}_1$-colorings such that the union of any two color classes induces a graph in
$\mathcal{P}_2$.  
Table~\ref{table-zilber}  also contains the definitions of 
{\em acyclic colorings, $t$-improper colorings} and {\em co-colorings},
which will be discussed in Section \ref{se:varia}.
\begin{table}[htb]
\begin{center}
\begin{tabular}{|l|c|c|}
\hline
$\Zilber$-coloring & $\mathcal{P}_1$ & $\mathcal{P}_2$\\
\hline
trivial & all graphs & all graphs\\
proper & edgeless graphs & all graphs\\
acyclic & edgeless graphs & forests\\
convex & connected graphs & all graphs\\
harmonious & edgeless graphs & at most one edge\\
${\rm mcc}_t$ & conn. cpts size $\leq t$ & all graphs\\
$DU(H)$ & disjoint union of $\cong H$ & all graphs\\ 
$t$-imp &  max. degree $t$ & all graphs\\
co-coloring & clique or edgeless & all graphs\\
$\mathcal{AH}$-coloring & $\mathcal{AH}$ & all graphs\\
\hline
\end{tabular}
\end{center}
\ifrevised
\caption{$\mathcal{P}_1$-colorings where the union of any two color classes is in $\mathcal{P}_2$.
In the last line $\mathcal{P}_1$ is an additive induced hereditary property 
(closed under taking induced subgraphs and disjoint unions).}
\fi 
\label{table-zilber}
\end{table}

\section{Detailed case study: Dichotomy theorems}
\label{se:explicit}
\subsection{Harmonious colorings}
\label{se:harmonious}

Recall that a coloring is
{\em harmonious} if it is a proper vertex coloring and
every pair of colors
occurs along some edge at most once, and $\chi_{harm}(G;X)$ is the corresponding graph polynomial.

\begin{proposition}
For every $k \in \N$ there is a polynomial time Turing machine $T(k)$ which computes
$\chi_{harm}(-;k)$. In other words $\N \subseteq \mathrm{EASY}(\chi_{harm})$ non-uniformly.
\end{proposition}
\begin{proof}
For $X=k$ a positive integer and a graph $G$ on $n$ vertices,
$\chi_{harm}(G;k) \neq 0$ implies that $G$ has  at most
$e(k)={k \choose 2 }$ 
edges.
Furthermore, there are $k^{2\cdot e(k)}$
colorings of $2 \cdot e(k)$ 
vertices with $k$ colors.
Let $i(G)$ be the number of isolated vertices of $G$,

$T(k)$ proceeds as follows:
\begin{enumerate}[(i)]
\item
Determine $|E(G)|$. If $|E(G)| \geq e(k)+1$ we have $\chi_{harm}(G;k) = 0$.
\item
Otherwise, we strip $G$ of all its isolated vertices to obtain $G'$,
which has at most $2\cdot e(k)$ vertices.  
\item
We count the colorings of $G'$ which are harmonious, i.e., $\chi_{harm}(G';k)$,
which takes time $t(k)$, 
independently of the number of vertices of $G$.
\item
Therefore $\chi_{harm}(G;k) = k^{i(G)} \cdot \chi_{harm}(G';k)$.
\end{enumerate}
It follows that $T(k)$ runs in time $O(n^2)$ where the constants depend on $k$.
\end{proof}

\begin{remark}
In spite of the low complexity of the above algorithms,
we cannot compute all the coefficients of $\chi_{harm}(G;X)$ in polynomial time.
To compute $\chi_{harm}(G;X)$ for a graph $G$ on $n$ vertices we would have to compute 
for $n+1$ values $k_1,\dots, k_{n+1}$ of $X$ 
the value of the function $\chi_{harm}(G;X)$ and then use Lagrange interpolation. However, 
the above algorithm inspects $k^{2 \cdot e(k)}$
colorings, which for at least one of the values $\chi_{harm}(G;k_i)$ is bigger than $n^{2 \cdot e(n)}$.
It follows from Theorem
\ref{th:harmonious} below that, indeed, $\chi_{harm}(G;X)$ cannot be computed in polynomial time
unless $\bP = \sharp \bP$.
\end{remark}

\begin{theorem}
\label{th:harmonious}
For each $a \in \F-\N$ the evaluation of  $\chi_{harm}(G;a)$  is $\sharp\bP$-hard.
\end{theorem}

\begin{proof}
Let $G$ be a graph. We form $S(G)$ in the following way (see Figure \ref{fig:graham}).
\\
We first form $G_1$ using $G$ by
adding a new  vertex $v_e$ for each edge $e=(u,v)$ of $G$.  Then replace the edge $e$ by two new edges
$(u,v_e)$ and $(v_e,v)$. 
\\
Using $G_1$ we now form $S(G)$:
We connect all the new vertices $v_e: e \in E(G)$ of $G_1$ 
such that they form a complete graph on $|E(G)|$ vertices.

For a graph $G$ and $k\in\mathbb N$,
\begin{equation}\label{eq1}
  \chi_{harm}(S(G);k+e(G))=\chi(G;k)\cdot\binom{k+e(G)}{e(G)}e(G)!,
  \end{equation}
where $e(G)=|E(G)|$ and $\chi(G;k)$ is the chromatic polynomial of $G$ evaluated at $k$.
Since Equation (\ref{eq1}) holds for every $k\in\mathbb N$, it is a polynomial identity,
which can be written as follows:
\begin{equation}\label{eq2}
\chi_{harm}(S(G);X)=X_{(e(n))}\cdot\chi(G;X-e(G)).
\end{equation}
Equation~\eqref{eq2} provides a polynomial time reduction from the coefficients of $\chi(G;X)$ to the coefficients of $\chi_{harm}(S(G);X)$, and vice versa.  
In particular, determining  $\chi_{harm}(S(G);X)$ is $\#{\mathbf P}$-hard.
We also see from Equation~\eqref{eq2} that for $a\in{\mathbf F}-\mathbb N$,
the graph parameter $\chi(G;a-e(G))$ is polynomial time equivalent to the evaluation $\chi_{harm}(S(G);a)$.
\begin{center}
\begin{figure}
\makebox{
\begin{tikzpicture}[scale=0.7]
\draw[thick] (0,0) -- (3,0) ;
\draw[thick] (0,0) -- (0,3) ;
\draw[thick] (0,0) -- (3,3) ;
\draw[thick] (0,3) -- (3,3) ;
\draw[thick] (3,0) -- (3,3) ;
\draw[fill=black,opaque] (0,0) circle (0.2cm) ; 
\draw[fill=black,opaque] (3,0) circle (0.2cm) ; 
\draw[fill=black,opaque] (0,3) circle (0.2cm) ; 
\draw[fill=black,opaque] (3,3) circle (0.2cm) ; 
\end{tikzpicture}
$G$
\hspace{1.0cm}
\begin{tikzpicture}[scale=0.7]
\draw[thick] (0,0) -- (3,0) ;
\draw[thick] (0,0) -- (0,3) ;
\draw[thick] (0,3) -- (3,3) ;
\draw[thick] (3,0) -- (3,3) ;
\draw[thick] (0,0) -- (1,2) ;
\draw[thick] (1,2) -- (3,3) ;
\draw[fill=black,opaque] (0,0) circle (0.2cm) ; 
\draw[fill=black,opaque] (3,0) circle (0.2cm) ; 
\draw[fill=black,opaque] (0,3) circle (0.2cm) ; 
\draw[fill=black,opaque] (3,3) circle (0.2cm) ;   
\draw[fill=red,opaque] (1.5,0) circle (0.2cm) ; 
\draw[fill=red,opaque] (0,1.5) circle (0.2cm) ; 
\draw[fill=red,opaque] (1,2) circle (0.2cm) ; 
\draw[fill=red,opaque] (1.5,3) circle (0.2cm) ; 
\draw[fill=red,opaque] (3,1.5) circle (0.2cm) ;   
\end{tikzpicture}
$G_1$
\hspace{1.0cm}
\begin{frame}{}
\begin{tikzpicture}[scale=0.7]
\draw[thick] (0,0) -- (3,0) ;
\draw[thick] (0,0) -- (0,3) ;
\draw[thick] (0,0) -- (1,2) ;
\draw[thick] (1,2) -- (3,3) ;
\draw[thick] (0,3) -- (3,3) ;
\draw[thick] (3,0) -- (3,3) ;
\draw[fill=black,opaque] (0,0) circle (0.2cm) ; 
\draw[fill=black,opaque] (3,0) circle (0.2cm) ; 
\draw[fill=black,opaque] (0,3) circle (0.2cm) ; 
\draw[fill=black,opaque] (3,3) circle (0.2cm) ;    
\draw[fill=red,opaque] (1.5,0) circle (0.2cm) ; 
\draw[fill=red,opaque] (0,1.5) circle (0.2cm) ; 
\draw[fill=red,opaque] (1,2) circle (0.2cm) ; 
\draw[fill=red,opaque] (1.5,3) circle (0.2cm) ; 
\draw[fill=red,opaque] (3,1.5) circle (0.2cm) ;     
\draw[red,thick] (1.5,0) -- (0,1.5) ;
\draw[red,thick] (1.5,3) -- (0,1.5) ;
\draw[red,thick] (1.5,3) -- (3,1.5) ;
\draw[red,thick] (1.5,0) -- (3,1.5) ;
\draw[red,thick] (1,2) -- (1.5,0) ;
\draw[red,thick] (1,2) -- (0,1.5) ;
\draw[red,thick] (1,2) -- (1.5,3) ;
\draw[red,thick] (1,2) -- (3,1.5) ;
\draw[red,thick] (1.5,0) -- (1.5,3) ;
\draw[red,thick] (0,1.5) -- (3,1.5) ;
\end{tikzpicture}
\end{frame}
$S(G)$
}
\caption{Constructing $S(G)$ from $G$}
\label{fig:graham}
\end{figure}
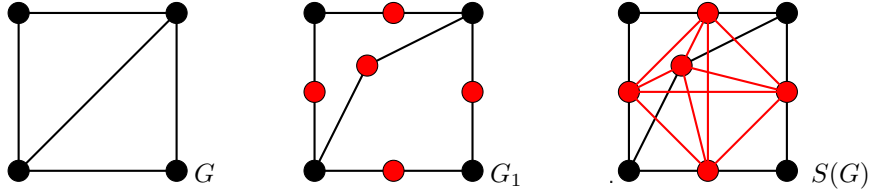
\end{center}

Finally, evaluating $\chi(G;a-e(G))$ is $\#{\mathbf P}$-hard for $a\not\in\mathbb N$. 
To see this, first observe the following polynomial identity by the multiplicativity of 
the chromatic polynomial over disjoint unions:
\begin{equation}\label{eq:harm3}
\begin{array}{l}
  \chi(G\sqcup K_{1,n};X\!-\!e(G\sqcup K_{1,n}))  = \\
  \chi(K_{1,n};X\!-\!e(G\sqcup K_{1,n})) \cdot \chi(G;X\!-\!e(G\sqcup K_{1,n})) = \\
  (X\!-\!e(G)\!-\!n)\cdot (X\!-\!e(G)\!-\!n\!-\!1)^n\cdot\chi(G;X\!-\!e(G)\!-\!n).
\end{array}
\end{equation}

For $a\in\mathbf F-\mathbb N$, we use Equation~\eqref{eq:harm3} to obtain the evaluations
$\chi(G;a\!-\!e(G)\!-\!n)$ for $n=0,1,\dots, |V(G)|$ 
and we apply Lagrange interpolation to compute $\chi(G;X\!-\!e(G))$.
By Equation~\eqref{eq2} this determines $\chi_{harm}(S(G);X)$, which as we have seen is a $\# {\mathbf P}$-hard problem.
Consequently evaluating $\chi_{harm}$ at $X=a$ is $\#{\mathbf P}$-hard. 
\end{proof}

\begin{remark}
From Equation~\eqref{eq1} we also see that, since evaluating $\chi(G;3)$ is $\#{\mathbf P}$-complete, 
we must indeed have non-uniform polynomial time evaluation of $\chi_{harm}(G;X)$ at positive integer points, 
as stated in Proposition 3.1.
\end{remark}
Hence we have shown:
\begin{theorem}
The Difficult Point Dichotomy is true for $\chi_{harm}(G;X)$ 
\\
with 
$\mathrm{EASY}(\chi_{harm}) =\N$ non-uniformly,
\ifrevised
and $\mathrm{\sharp PHARD}(\chi_{harm}) = \F -\N$.
\fi 
\end{theorem}


\subsection{Convex colorings}
\label{se:convex}
Recall that
a {\em convex coloring} of a graph $G$ is an assignment of colors to
its vertices so that for each color $c$ the subgraph of $G$
induced by the vertices receiving color $c$ is connected\footnote{
We consider a graph with no vertices to be connected.
}.
The resulting graph polynomial is $\chi_{convex}(G;X)$.

The following is easily verified:
\begin{proposition}
\label{prop:convex}
\begin{enumerate}[(i)]
\item
For $X=1$ we have
$$
\chi_{convex}(G;1)=
\begin{cases}
1 & \mbox{if   } G \mbox{   is   connected} \\
0 & \mbox{else}.
\end{cases}
$$
\item
For $k \in \N^+$ we have
$$
\chi_{convex}(G \sqcup K_1;k) = k \cdot \chi_{convex}(G;k-1).
$$
\end{enumerate}
\end{proposition}
In \cite{mak:Question}, it was asked whether $\chi_{convex}(-;2)$ is $\sharp\bP$-hard.
The question was answered in the positive and posted in \cite{pr:GoodallNoble2008}, but remained unpublished.
Note that the number of convex colorings using at most two
colors is equal to zero if $G$ has three or more connected
components and equal to two if $G$ has exactly two connected
components. So we may restrict our attention to connected graphs.

\begin{theorem}[A. Goodall and S. Noble, \cite{pr:GoodallNoble2008}]
\label{th:goodall-noble}
Evaluating
$\chi_{convex}(G;X)$ for $X=2$ on connected graphs is $\sharp\bP$-complete.
\end{theorem}
The proof is given in Section \ref{se:GN}.
Combining Proposition \ref{prop:convex} and Theorem \ref{th:goodall-noble} we get the following:
\begin{theorem}
\label{th:convex}
$\mathrm{EASY}(\chi_{convex}) = \{0,1\}$ and $\mathrm{\sharp PHARD}(\chi_{convex})= \F - \{0,1\}$.
\end{theorem}

\subsection{$DU(H)$-colorings}
\label{se:triangles}
Let $H$ be a fixed connected graph. 
A coloring $f: V(G) \rightarrow [k] $ 
is an {\em $DU(H)$-coloring} with $k$ colors,
if each color class induces a disjoint union of copies of $H$.
We denote by $\chi_{DU(H)}(G; k)$ the number of 
$DU(H)$-colorings of $G$ with at most $k$ colors,
and by $\hat{\chi}_{DU(H)}(G; k)$ the number of 
$DU(H)$ colorings of $G$ with exactly $k$ colors.

We easily verify:
\begin{proposition}
\begin{enumerate}[(i)]
\item
Being a $DU(H)$-coloring with $k$ colors
is a $\Zilber$-property, hence
$\chi_{DU(H)}(G; k)$ 
is a  polynomial in $k$.
However, $\hat{\chi}_{DU(H)}(G; k)$ is not a $\Zilber$-property, and in general is not a polynomial in $k$.
\item
\ifrevised 
$\chi_{DU(H)}(G; k) = \sum_{i=1}^k {k \choose i} \hat{\chi}_{DU(H)}(G; i)$.
\fi 
\item
$\chi_{DU(H)}(G; k)$ is multiplicative over disjoint unions.
\item
For $k=1$, a graph $G$ is $DU(H)$-colorable  iff $G$ is a disjoint union of
$H$s.
\item
For $n(G)\not \equiv 0 \mod{n(H)}$ the polynomial $\chi_{DU(H)}(G; k)$ vanishes.
\end{enumerate}
\end{proposition}

Let $v \in V(H)$. We define $\Box_{H,v}(G)$ to be the graph with vertex set $V(G) \sqcup V(H)$,
and edge set $E(G) \sqcup E(H) \sqcup V(G) \times \{v\}$.
We can apply Linial's Trick, cf. Lemma \ref{LTrick}, to analyze the complexity of $\chi_{DU(H)}(G;a)$. 
\begin{proposition}
Let $H$ be a connected graph.
\begin{enumerate}[(i)]
\item
$ \chi_{DU(H)}(\Box_{H,v}(G);k) = k \cdot \chi_{DU(H)}(G;k-1)$.
\item
For every $a,b \in \N$ and $b > a$,
$\chi_{DU(H)}(G;a)$ is polynomial time reducible to
$ \chi_{DU(H)}(G;b)$.
\item
For every $a_0 \in \F-\N$, computing the coefficients of
$\chi_{DU(H)}(G;X)$ is polynomial time reducible to $\chi_{DU(H)}(G;a_0)$.
\end{enumerate}
\end{proposition}
The proof is the same as in \cite{ar:Linial86}. For the convenience of the reader we sketch it here.
\begin{proof}
(i) 
All the vertices  of $V(H)$
have to be colored by the same color but differently from the vertices in $V(G)$. 
\\
(ii) Apply (i) $b-a$ many times.
\\
(iii) Let $G_0 = G$, $G_{i+1} =\Box_{H,v}(G_i)$. Using $\chi_{DU(H)}(-;a_0)$ we can compute
$\chi_{DU(H)}(G_i;a_0)$ for sufficiently many $i$'s and then use Lagrange Interpolation to compute
the coefficients of $\chi_{DU(H)}(G;X)$.
\end{proof}

Related  decision and counting problems have been considered in the literature. 

In the following $\alpha$ is a nonnegative integer.
\begin{framed}
\begin{trivlist}\item[]
\textbf{Problem:} $\mathrm{CliqueCover_{\alpha}}$\\
\textbf{Input:} Graph $G$. If $\alpha \geq 1$ then $n(G) ={\alpha} \cdot m$.\\
\textbf{Question:} 
Can we partition $V(G)$ into sets $V_i$ such that each $V_i$ 
induces a clique (for $\alpha =0$) or induces a copy of $K_{\alpha}$ (for $\alpha \geq 1$)?
\end{trivlist}
\end{framed}

\begin{framed}
\begin{trivlist}\item[]
\textbf{Problem:} $\sharp\mathrm{CliqueCover_{\alpha}}$\\
\textbf{Input:} Graph $G$. If $\alpha \geq 1$ then $n(G) ={\alpha} \cdot m$.\\
\textbf{Output:}  The number of partitions of $V(G)$ into sets $V_i$ such that each $V_i$ 
induces a clique (for $\alpha =0$) or induces a copy of $K_{\alpha}$ (for $\alpha \geq 1$).
\end{trivlist}
\end{framed}

We note for $\alpha =0$ this is the classical clique cover decision problem of  
Richard Karp's original 21 problems, see \cite{bk:GJ}.

From the literature we know the following:

\begin{theorem}
\begin{enumerate}[(i)]
\item
$\mathrm{CliqueCover_{0}}$ is $\bNP$-complete and $\sharp\mathrm{CliqueCover_{0}}$
is $\sharp\bP$-complete, \cite{bk:GJ,ar:Hunt-etal-1998}.
\item
$\mathrm{CliqueCover_{1}}$ and $\sharp\mathrm{CliqueCover_{1}}$ are both trivial.
\item
$\mathrm{CliqueCover_{2}}$ is the same as finding a perfect matching, and is in $\bP$,
and $\sharp\mathrm{CliqueCover_{2}}$  is
$\sharp \bP$-complete by \cite{ar:valiant-SIAM}.
\item
$\mathrm{CliqueCover_{3}}$ is $\bNP$-complete and 
$\sharp\mathrm{CliqueCover_{3}}$
is $\sharp\bP$-complete, \cite{ar:KirkpatrickHell-1983,ar:Hunt-etal-1998}.
\end{enumerate}
\end{theorem}

\begin{problem}
\label{prob-6}
For which $\alpha \geq 4$ is
$\sharp\mathrm{CliqueCover_{\alpha}}$ $\sharp\bP$-complete?
\end{problem}

The connection between 
$\sharp\mathrm{CliqueCover_{\alpha}}$
and $\chi_{DU(K_{\alpha})}(G;X)$
is given as follows:

\begin{proposition}
\begin{enumerate}[(i)]
\item
$\chi_{DU(K_1)}(G;X) = \chi(G;X)$, the chromatic polynomial.
\item
$\hat{\chi}_{DU(K_{\alpha})}(G;\frac{n}{\alpha}) = \left(\frac{n}{\alpha}\right)! \cdot \sharp\mathrm{CliqueCover_{\alpha}}$ 
where $\alpha \neq 0$ and 
$n(G)$ is divisible by $\alpha$.
\end{enumerate}
\end{proposition}
\begin{proof}
If $n = n(G)$ is divisible by $\alpha$, the maximal value $k$ such that
$\hat{\chi}_{DU(K_{\alpha})}(G;k) \neq 0$ is $k= \frac{n}{\alpha}$.
In this case the subgraph induced by each color class is isomorphic to $K_{\alpha}$,
and there are $\left(\frac{n}{\alpha}\right)!$ 
many such colorings.
\end{proof}

\newcommand{\alphaOfTwoAlpha}{\sharp\alpha\mbox{-}\mathrm{of}\mbox{-}2\alpha\mbox{-}\mathrm{SAT}}
We will prove that $\chi_{DU(K_\alpha)}(G;2)$ is $\sharp\bP$-hard for every $\alpha\geq 2$
in Theorem \ref{th:DU-hard} below by describing a polynomial time reduction
from the $\sharp\bP$-hard problem 
$\alphaOfTwoAlpha$.
From \cite{ar:CreignouHermann96} we know that
$\alphaOfTwoAlpha$ is $\sharp \bP$-complete.

\begin{framed}
\begin{trivlist}\item[]
\textbf{Problem:} $\alphaOfTwoAlpha$ \\
\textbf{Input:} $2\alpha$CNF formula $\Theta$. \\
\textbf{Output:} The number of truth assignments such that, in each clause, $\alpha$ literals
are true and the other $\alpha$ are false. 
\end{trivlist}
\end{framed}

Now we proceed to describe the construction of a graph $\overline{G}_\Theta$ for every  
input $\Theta$ to $\alphaOfTwoAlpha$. We will prove that  
$\alphaOfTwoAlpha(\Theta)=\chi_{DU(K_\alpha)}(\overline{G}_\Theta;2)$. 

Let $\Theta = \bigwedge_{i=1}^s C_i$ be 
a $2\alpha$CNF formula with clauses $C_i = l_{i,1} \lor \cdots \lor l_{i,2\alpha}$. 
Without loss of generality, let the variables which occur in $\Theta$ be $x_t: t\in [r]$. 
The literals $l_{i,j}$ of $\Theta$ are either variables $x_t$ or negation of variables $\neg x_t$. 
We define $\mathrm{var}(x_t) = \mathrm{var}(\neg x_t) = x_t$.

For every clause $C_i$ of $\Theta$,
let $G_i$ be a clique of size $2\alpha$ whose vertices are labeled by 
$l_{i,1},\ldots,l_{i,2\alpha}$ respectively. 
Let $G_\Theta$ be the disjoint union of $G_i$ for all clauses $C_i$ of $\Theta$. 
For every variable $x_t:t\in [r]$ let $D_t$ be a clique of size $2\alpha$ in which $\alpha$ vertices are labeled $x_t$
and the other $\alpha$ vertices are labeled $\neg x_t$. 
Let $D_\Theta$ be the disjoint union of $D_t: t\in [r]$.

We construct a graph $\overline{G}_\Theta$ as follows. 
The graph $\overline{G}_\Theta$ is obtained from the disjoint union $G_\Theta \sqcup D_\Theta$ of $G_\Theta$ and $D_\Theta$ 
by
adding an edge between any vertex of $G_\Theta$
and any vertex of $D_{\Theta}$ whose labels are negations of each other
(i.e. one is $x_t$ and the other is $\neg x_t$, $t\in [r]$). 

Any coloring $c:V(\overline{G}_\Theta)\to [2]$ can be interpreted
as assigning the truth values \textit{true} (for the color $1$) and \textit{false} (for the color $2$) to the literals
labeling the vertices. 
A coloring $c$ of $\overline{G}_\Theta$ is \textit{consistent} if the truth values
assigned by $c$ to the literals induce 
a well-defined truth value assignment $as_c$ to the variables. 
More precisely, $c$ is consistent if
any two vertices labeled the same (both $x_t$ or both $\neg x_t$, $t\in [r]$)
have the same color 
and any two vertices with opposing labels ($x_t$ and $\neg x_t$, $t\in [r]$) have different colors.

\begin{lem}\label{lem:consistent-colorings}
Let $c:V(\overline{G}_\Theta) \to [2]$ be a coloring of $\overline{G}_\Theta$. 
\begin{enumerate}[(i)]
\item If $c$ is a $DU(K_\alpha)$-coloring, then $c$ is consistent. 
\item If $c$ is consistent, then the following are equivalent:
\begin{enumerate}[(a)]
 \item $c$ is a $DU(K_\alpha)$-coloring.
 \item $as_c$ satisfies the condition of $\alphaOfTwoAlpha$. 
\end{enumerate}
\end{enumerate}
\end{lem}
\begin{proof}
For (i), we assume $c$ is a $DU(K_\alpha)$-coloring.
For all $i$ and $t$, $G_i$ and $D_t$ are cliques of size $2\alpha$. 
Hence, in each of the cliques $G_i$ and $D_t$ exactly $\alpha$ vertices are colored $1$ and the other
$\alpha$ vertices are colored $2$. As a consequence, no vertex $u$ 
is adjacent to any vertex $v$ such that $c(u)=c(v)$
and  $u$ and $v$ belong to different cliques in $G_\Theta \sqcup D_\Theta$. 

Let $u$ be a vertex of $G_\Theta$ such that $u$
is labeled by $l$ and $\mathrm{var}(l)=x_t$. Since there are edges between $u$
and the $\alpha$ vertices labeled by the negation of $l$ in $D_t$, these $\alpha$ vertices cannot have the color $c(u)$.
As a consequence, the $\alpha$ vertices labeled by $l$ in $D_t$ must have the color $c(u)$. 
We get that all the vertices of $\overline{G}_\Theta$ labeled with $l$ receive the same color,
which is different from the color of the vertices labeled by the negation of $l$, and we get (i).
Moreover, $c$ assigns exactly $\alpha$ vertices in each $G_i$ to each of the colors, 
which implies that $as_c$ assigns exactly $\alpha$ of the literals of $C_i$ to \textit{true}
and the other $\alpha$ to \textit{false}.
Hence $as_c$ is counted by $\alphaOfTwoAlpha$, and we get the direction (a)$\Rightarrow$(b) of (ii). 

For the direction (b)$\Rightarrow$(a) of (ii), let $c$ be a consistent coloring 
such that $as_c$ is counted by $\alphaOfTwoAlpha$.
We will prove that $c$ is a $DU(K_\alpha)$-coloring. 
Since $as_c$ is counted by $\alphaOfTwoAlpha$, 
it assigns \textit{true} to $\alpha$ literals and
\textit{false} to the other $\alpha$ literals of each clause. 
Hence, $c$ colors every clique $G_i$ so that $\alpha$ vertices receive color $1$ and the other $\alpha$ receive color $2$.
Each of the two colors induces a clique of size $\alpha$ in $G_i$. 
Since $as_c$ is consistent and $D_t$ consists of $\alpha$ vertices  labeled by some label $l$ and 
$\alpha$ vertices labeled by the negation of $l$, each color class of $c$ induces a clique of size $\alpha$ in $D_t$. 
It remains to notice that, since $c$ is consistent, any other edge of $\overline{G}_\Theta$ crosses between the color classes,
hence does not belong to the induced subgraph of any of the two colors. Consequently, each of the colors 
induces a disjoint union of cliques of size $\alpha$ in $\overline{G}_\Theta$. 

\end{proof}

As a consequence of Lemma \ref{lem:consistent-colorings} 
we have $\alphaOfTwoAlpha(\Theta) = \chi_{DU(K_\alpha)}(\overline{G}_\Theta;2)$. 

From \cite{ar:CreignouHermann96} we have:
\begin{theorem}\label{thm:CH96-alphaOfTwoAlpha}
For every $\alpha\geq 2$, $\alphaOfTwoAlpha$ is $\sharp\bP$-hard. 
\end{theorem}

From Lemma \ref{lem:consistent-colorings} and Theorem \ref{thm:CH96-alphaOfTwoAlpha}, 
and using the fact that the construction of $\overline{G}_\Theta$ can be done in polynomial time, we get:
\begin{theorem}\label{th:DU-hard}
For every $\alpha \geq 2$, 
$\chi_{DU(K_\alpha)}(G;2)$ is $\sharp\bP$-hard. 
\end{theorem}

\begin{proposition}\label{prop:DU-easy}
$\chi_{DU(K_\alpha)}(G;0)$ and
$\chi_{DU(K_\alpha)}(G;1)$ 
are polynomial time computable.
\end{proposition}
\begin{proof}
For $X=0$, $\chi_{DU(K_\alpha)}(G;0)$ is always $0$ and hence trivially polynomial time computable. 
For $X=1$, there is exactly one coloring $c:V(G)\to[1]$ and $c$ is a $DU(K_\alpha)$-coloring 
iff $G$ is a disjoint union of copies of $K_\alpha$, which can be verified in polynomial time. 
\end{proof}

Putting all this together we get the full complexity spectrum for
$$\chi_{DU(K_{\alpha})}(G;X)$$ for $\alpha \geq 2$. 
Recall that 
$\chi_{DU(K_{1})}(G;X) =\chi(G;X)$ 
is the chromatic polynomial.

\begin{theorem}
\label{th:cliquealpha}
For all $\alpha \geq 2$, we have
$\mathrm{EASY}(\chi_{DU(K_{\alpha})}) = \{0,1\}$ and
\\
$\mathrm{\sharp PHARD}(\chi_{DU(K_{\alpha})}) = \F - \{0,1\}$.
\\
Moreover,  for $\alpha \geq 1$
the Difficult Point Dichotomy is true for $\chi_{DU(K_{\alpha})}(G;X)$,
as for $\alpha=1$ it includes the chromatic polynomial.
\end{theorem}

\section{Counting convex colorings is $\sharp \bP$-complete: the proof}
\label{se:cproof}
\label{se:GN}
The purpose
of this section is to prove Theorem \ref{th:goodall-noble}.
This answers a question originally asked in
\cite{mak:Question}, 
whether the problem of counting the number
of convex colorings using at most two colors, 
i.e. computing $\chi_{convex}(-;2)$, 
is $\#$P-complete on connected graphs.

This section follows almost verbatim the preprint posted as \cite{pr:GoodallNoble2008}.

\ifrevised
\subsection{Cuts, crossing sets, and cocircuits}
\fi
Let $X$ and $Y$ be disjoint sets of vertices of a graph $G$. The set
of edges of $G$ that have one endpoint in $X$
and the other in $Y$ is denoted by $\delta(X,Y)$. Given a connected graph $G$, a cut is a
partition of $V(G)$ into two (non-empty) sets called its
\emph{shores}. The \emph{crossing set} of a cut with shores $X$
and $Y$ is $\delta(X,Y)$. A cut is a \emph{cocircuit} if no proper
subset of its crossing set is the crossing set of a cut. 
Given a set $A$ of edges, we denote by $G\setminus A$ the graph
obtained by removing the edges in $A$ from $G$. 
We will use the following observation:
\begin{lem}
Let $G$
be a connected graph. Then a cut of $G$ with crossing set $A$ is a
cocircuit if and only if $G\setminus A$ has exactly two connected
components.
\end{lem}

Note that our terminology is slightly at odds with standard usage
in the sense that the terms cut and cocircuit usually refer to
what we call the crossing set of respectively a cut and a
cocircuit. Our usage prevents some cumbersome descriptions in the
proofs. We will however abuse our notation by saying that a cut or
cocircuit has size $k$ if its crossing set has size $k$.

\ifrevised
\subsection{Reductions}
We prove Theorem~\ref{th:goodall-noble} by a sequence of reductions
involving the following problems:
\fi
\begin{framed}
\begin{trivlist}\item[]
\textbf{Problem:} $\#$\textsc{Cocircuits}\\
\textbf{Input:} Connected graph $G$.\\
\textbf{Output:} The number of cocircuits of $G$.
\end{trivlist}
\end{framed}

\begin{framed}
\begin{trivlist}\item[]
\textbf{Problem:} $\#$\textsc{Required Size Cocircuits}\\
\textbf{Input:} Connected graph $G$, strictly positive integer $k$.\\
\textbf{Output:} The number of cocircuits of $G$ of size $k$.
\end{trivlist}
\end{framed}

\begin{framed}
\begin{trivlist}\item[]
\textbf{Problem:} $\#$\textsc{Max Cut}\\
\textbf{Input:} Connected graph $G$, strictly positive integer $k$.\\
\textbf{Output:} The number of cuts of $G$ of size $k$.
\end{trivlist}
\end{framed}

\begin{framed}
\begin{trivlist}\item[]
\textbf{Problem:} $\#$\textsc{Monotone 2-SAT}\\
\textbf{Input:} A Boolean formula in conjuctive normal form in
which each clause contains two variables and there are no negated
literals.\\
\textbf{Output:} The number of satisfying assignments.
\end{trivlist}
\end{framed}
It is easy to see that each of these problems is a member of
$\#$P.

The following result is from Valiant's seminal paper on
$\#$P~\cite{ar:valiant-SIAM}.
\begin{theorem}~\label{th:2SAT}
$\#$\textsc{Monotone 2-SAT} is $\#P$-complete.
\end{theorem}

We will establish the following reductions.
\begin{trivlist}\item[]
\begin{center}
$\#$\textsc{Monotone 2-SAT }$\propto\#$\textsc{Max
Cut }$\propto\#$\textsc{Required Size Cocircuits}\\
$\propto\#$\textsc{Cocircuits }$\propto$
\ifrevised
$\chi_{convex}(-;2)$
\fi
\end{center}
\end{trivlist}
Combining Theorem~\ref{th:2SAT} with these reductions shows that
each of the five problems that we have discussed is
$\#$P-complete. As far as we are aware, each of these reductions
is new. We have not been able to find a reference showing that
$\#$\textsc{Max Cut} is $\#P$ complete. Perhaps it is correct to
describe this result as `folklore'. In any case we provide a proof
below. Some similar problems, but not exactly what we consider
here, are shown to be $\#P$ complete 
in~\cite{prba:83}.

\begin{lem}
$\#$\textsc{Monotone 2-SAT} $\propto \#$\textsc{Max Cut}.
\end{lem}
\begin{proof}
Suppose we have an instance $I$ of $\#$\textsc{Monotone 2-SAT}
with variables $x_1,\ldots,x_n$ and clauses $\mathcal C=\{
C_1,\ldots,C_m\}$. We construct a corresponding instance $M(I)=(G,k)$ of
$\#$\textsc{Max Cut} by first defining a graph $G$ with vertex set
\[ \{x\} \cup \{x_1,\ldots,x_n\} \cup  \bigcup \{\{c_{i,1},\ldots,c_{i,6}\}:1\leq i \leq m\}\]
For each clause we add nine edges to $G$. Suppose $C_j$ is $x_u
\vee x_v$. Then we add the edges
\[ xc_{j,1}, c_{j,1}c_{j,2}, c_{j,2}x_u, x_uc_{j,3}, c_{j,3}c_{j,4}, c_{j,4}x_v, x_vc_{j,5}, c_{j,5}c_{j,6}, c_{j,6}x.\]
Distinct clauses correspond to pairwise edge-disjoint circuits, each
of size $9$. Now let $k =
8|\mathcal C|$. We claim that the number of solutions of instance
$M(I)$ of \#{\sc Max Cut} is
equal to $2^{|\mathcal C|}$ times the number of
satisfying assignments of $I$.


Given a solution of $I$, let $L_1$ be the set of variables assigned
the value true and $L_0$ the set of variables assigned false
together with $x$. Observe that for each clause $C_j=x_u\vee x_v$
there are three choices of how to add the vertices
$c_{j,1},\ldots,c_{j,6}$ to either $L_0$ or $L_1$ so that exactly
eight edges of the circuit corresponding to $C_j$ have one
endpoint in $L_0$ and the other in $L_1$. Clearly the choices for
each clause are independent and distinct satisfying assignments
result in distinct choices of $L_0$ and $L_1$. Any of the choices
of $L_0$ and $L_1$ constructed in this way may be taken as the
shores of a cut of size $8|\mathcal C|$. Hence we have constructed
$3^{|\mathcal C|}$ solutions of $M(I)$ corresponding to each
satisfying assignment of $I$.


In any graph the intersection of a set of edges forming a circuit
and a crossing set of a cut must always have even size. So in a
solution of $M(I)$ each of the edge-disjoint circuits making up
$G$ and corresponding to clauses of $I$ must contribute exactly
eight edges to the cut. Suppose $U$ and $V\setminus U$ are the
shores of a cut of $G$ of size $8|\mathcal C|$. Then it can easily
be verified that for any clause $C=x_u \vee x_v$ both $U$ and
$V\setminus U$ must contain at least one element from
$\{x_u,x_v,x\}$. So it is straightforward to see that this
solution of $M(I)$ is one of those constructed above corresponding
to the satisfying assignment where a variable is false if and only
if the corresponding vertex is in the same set as $x$.
\end{proof}

\begin{lem}
$\#$\textsc{Max Cut} $\propto \#$\textsc{Required Size
Cocircuits}.
\end{lem}
\begin{proof}
Suppose $(G,k)$ is an instance of $\#$\textsc{Max Cut}. We construct an instance $(G',k')$ of
$\#$\textsc{Required Size Cocircuits} as follows. Suppose $G$ has
$n$ vertices. To form $G'$ add new vertices
$x,x',x_1,\ldots,x_{n^2}$ to $G$. Now add an edge from $x$ to
every other vertex of $G'$ except $x'$ and similarly add an edge
from $x'$ to every other vertex of $G'$ except $x$. Let
$k'=n^2+n+k$. From each solution of the $\#$\textsc{Max Cut}
instance $(G,k)$ we construct $2^{n^2+1}$ solutions of the
$\#$\textsc{Required Size Cocircuits} instance $(G',k')$. Suppose
$C=(U,V(G)\setminus U)$ is a solution of $(G,k)$ then we may
freely choose to add $x,x',x_1,\ldots,x_{n^2}$ to either $U$ or
$V(G)\setminus U$, with the sole proviso that $x$ and $x'$ are not
both added to the same set, to obtain a cut in $G'$ of size
$k'=n^2+n+k$. Furthermore this cut is a cocircuit because both
shores contain exactly one of $x$ and $x'$ and so they induce
connected subgraphs.

Conversely suppose $C=(U,V(G')\setminus U)$ is a cocircuit in $G'$
of size $k'$. Consider the pair of edges incident with $x_j$. Note
that the partition $(x_j,V(G') \setminus x_j)$ is a cocircuit. So
if both of the edges incident with $x_j$ are in the crossing set
of $C$ then because of its minimality we must have $C=(x_j,V(G')
\setminus x_j)$ which is not possible because $C$ would then have
size $2<k'$. Now suppose that neither edge incident with $x_j$ is
in the crossing set of $C$. 

Then both $x$ and $x'$, and hence $x_1, \ldots, x_{n^2}$,
lie in the same
block of the partition constituting $C$. But since $G$ is a simple
graph, the maximum possible size of such a cocircuit is at most
$2n+\binom n 2 < n^2+n+k$. Hence precisely one of the edges
adjacent to $x_j$ is in the crossing set. So $x$ and $x'$ are in
different shores of $C$. Hence the crossing set of $C$ contains:
for each $j$ precisely one edge incident to $x_j$ ($n^2$ edges in
total), for each $v \in V(G)$ precisely one of edges $vx$ and
$vx'$ ($n$ edges in total) and $k$ other edges with both endpoints
in $V(G)$. So the partition $C'=(U\cap V(G),V(G)\setminus U)$ is a
cut of $G$ of size $k$ and hence $C$ is one of the cocircuits
constructed in the first part of the proof. Consequently the
number of solutions of the instance $(G',k')$ of
$\#$\textsc{Required Size Cocircuits} is $2^{n^2+1}$ multiplied by
the number of solutions of the instance $(G,k)$ of $\#$\textsc{Max
Cut}
\end{proof}

\begin{lem}
$\#$\textsc{Required Size
Cocircuits }$\propto\#$\textsc{Cocircuits}
\end{lem}
\begin{proof}
Given a graph $G$ let $N_k(G)$ denote the number of cocircuits of
size $k$ and $N(G)$ denote the total number of cocircuits. Let
$G_l$ denote the $l$-stretch of $G$, that is, the graph formed
from $G$ by replacing each edge of $G$ by a path with $l$ edges.
Let $m=|E(G)|$. Then we claim that
\[ N(G_l) =\sum_{k=1}^{m} l^kN_k(G) + \binom l 2 m.\]

To see this, suppose that $C$ is a cocircuit of $G_l$. If the
crossing set of $C$ contains two edges from one of the paths
corresponding to an edge of $G$ then by the minimality of the
crossing set of $C$ we see that it contains precisely these two
edges. The number of such cocircuits is $\binom l 2 m$.

Otherwise the crossing set $C$ contains at most one edge from each
path in $G_l$ corresponding to an edge of $G$. Suppose the
crossing set of $C$ contains $k$ such edges. Let $A$ denote the
corresponding edges in $G$. Then $A$ is the crossing set of a
cocircuit in $G$ of size $k$. From each such cocircuit we can
construct $l^k$ cocircuits of $G_l$ by choosing one edge from each
path corresponding to an edge in $A$. The claim then follows.

If we compute $N(G_1),\ldots,N(G_{m})$ then we may retrieve
$N_1(G),\ldots,N_{m}(G)$ by using Gaussian elimination because
the matrix of coefficients of the linear equations is an invertible Vandermonde matrix. The fact that the
Gaussian elimination may be carried out in polynomial time follows
from~\cite{edmonds:linalg}. \end{proof}

\begin{lem}
$\#$\textsc{Cocircuits }$\propto$
\ifrevised
$\chi_{convex}(-;2)$
\fi
\end{lem}
\begin{proof}
The lemma is easily proved using the following observation. When
two colors are available, there are two convex colorings of a
connected graph using just one color and the number of convex
colorings using both colors is equal to twice the number of
cocircuits.
\end{proof}

The preceding lemmas imply our main result, Theorem \ref{th:goodall-noble},
that 
\ifrevised
$\chi_{convex}(-;2)$
\fi
is $\#$P-hard.

%
\section{Detailed case study: Discrete spectra}
\label{se:explicit-1}
In this section we dicuss cases where we were not able to find a suitable version of Linial's Trick,
but where we could determine the complexity of the evaluations for non-negative integers.
\subsection{Proper edge colorings}
\label{se:edgecolorings}

Recall that $\chi_{edge}(G;k)$ counts the number of proper edge colorings of a graph $G$
with $k$ colors. It is a polynomial in $k$ because
$$
\chi_{edge}(G;k) = \chi(L(G);k),
$$
where $L(G)$ is the line graph of $G$ and $\chi(G;k)$ is the chromatic polynomial.
We have
$\chi_{edge}(G;0) = 0$ and
$$\chi_{edge}(G;1)= 
\begin{cases}
1 & \mbox{if   } G \mbox{ consists of isolated edges and vertices}\\
0 & \mbox{otherwise}
\end{cases}
$$
Although $\chi_edge(G;k)=\chi(L(G);k)$, where $L(G)$ is the line graph of $G$, not
every graph $G$ is the line graph of some graph $G'$.

The class of all finite line graphs, $\mathcal{LG}$, has been completely characterized,
\cite{beineke1970characterizations,bk:BrandstaedtLeSpinrad}.
\begin{theorem}[Beineke, 1970, \cite{beineke1970characterizations}]
There are nine graphs 
$F_i: 1 \leq i \leq 9$, 
each with at most $6$ vertices, 
such that
$G \in \mathcal{LG}$ if and only if no $F_i$ is an induced subgraph of $G$.
\end{theorem}
The complexity spectrum of the 
chromatic polynomial restricted to the class $\mathcal{LG}$
has, to the best of our knowledge, not been studied.

\ifskip
\else
Linial's Trick is straightforward:
Let 
$v = |V(G)|$ and $e = |E(G)|$. 
Consider the graphs $G_1 =G \bowtie K_1$ and $G_2 \bowtie K_2$. 
$G_1$ has one new vertex and $v$ new edges. Hence
$G_2$ has $v_2$ new vertices and $2\cdot v +1$ new edges. Hence
$$
\chi_{edge}(G \bowtie K_1;k + v) ) =  v! \cdot  \chi_{edge}(G;k)
$$
or alternatively
$$
\chi_{edge}(G \bowtie K_2;k +( 2\cdot v +1) ) = (2\cdot v +1)! \cdot  \chi_{edge}(G;k)
$$
We conclude that evaluating $\chi_{edge}(G;a)$ is equally difficult for all $a \in \C -\{0,1\}$.
\fi 

Surprisingly, the complexity of counting proper edge colorings was proven $\sharp\bP$-hard
only recently, \cite{pr:CaiGuoWilliams2014}:
\begin{theorem}[J. Y. Cai, H. Guo, T. Williams, 2014]
\label{th:CaiGuoWilliams}
For $k \in \N$ we have:
\begin{enumerate}[(i)]
\label{th:edgecoloring}
\item
\ifrevised
$\chi_{edge}(G;k)$ 
\fi 
is $\sharp\bP$-hard over planar $r$-regular graphs
for all $k \geq r \geq 3$.
\item
$\chi_{edge}$ is trivially tractable when $k \geq r \geq 3$ does not hold.
\end{enumerate}
\end{theorem}

The proof given in 
\cite{pr:CaiGuoWilliams2014}
reduces $\chi_{edge}(G;k)$ to computation of the diagonal
of the Tutte polynomial $T(G;X,X)$
using several intermediate steps via holants\footnote{
For background on holants, cf.
\cite{ar:Valiant2008,ar:CaiLuXia2011}}.
\begin{problem}
\label{prob-4}
Find an elementary (holant-free) proof of Theorem \ref{th:edgecoloring}.
\end{problem}

Theorem \ref{th:CaiGuoWilliams} gives us the discrete complexity spectrum for $k\in \N$.

We were unable to adapt Linial's Trick to proper edge colorings.
Therefore we do not know how to determine the complexity
of $\chi_{edge}(G;X)$ for $X=a$ and $a \in \C- \N$ or even $a \in \Q- \N$.

\begin{problem}
\label{prob-5}
\label{p:edgecolorings}
Determine the full complexity spectrum of $\chi_{edge}(G;X)$ for $X=a$ and $a \in \Q$ or $a \in \C$.
\end{problem}

\subsection{$mcc_t$-colorings}
\label{se:mcc}
Let $t \in \N$.
Recall that a 
coloring $f: V(G) \rightarrow [k]$ is an {\em $mcc_t$ -coloring} with $k$ colors,
if the connected components of each color class have at most $t$ vertices.

We easily verify:
\begin{proposition}
\label{prop:mcc-basic}
\begin{enumerate}[(i)]
\item
For fixed $t \in \N^+$
being an $mcc_t$ -coloring with $k$ colors
is a $\Zilber$-property, hence
$\chi_{mcc_t}(G; k)$ is a polynomial in $k$ (but not in $t$).
\item
$\chi_{mcc_t}(G; k)$ is multiplicative over disjoint unions.
\item
For $t=1$ we have $\chi_{mcc_1}(G; k)= \chi(G; k)$, i.e., it is the chromatic polynomial.
\item
For $k=1$ a graph $G$ is $mcc_t$ -colorable  iff $G$ is a disjoint union of connected graphs $H$  with  at most
$t$ vertices.
\end{enumerate}
\end{proposition}

We next establish a complexity result.
\begin{theorem}
\label{th:mcc_2}
Computing $\chi_{mcc_t}(G; 2)$ is $\sharp\bP$-complete for $t \geq 2$.
\end{theorem}
\ifskip\else
\documentclass[11pt,a4paper]{article}

\usepackage{fullpage}
\usepackage{amsmath,amsthm}
\usepackage{latexsym,amssymb}
\usepackage[T1]{fontenc}
\usepackage{times}
\usepackage{subfig}
\usepackage{graphicx}

\theoremstyle{plain}
\newtheorem{theorem}{Theorem}

\begin{document}
\fi 

\newcommand\nae{\mathrm{NAE}}
\newcommand\sat{\mathrm{SAT}}
\newcommand\sharpP{\sharp\bP}

\begin{figure}
  \centering
  \subfloat[Clause gadget]%
  {\includegraphics[scale=.7]{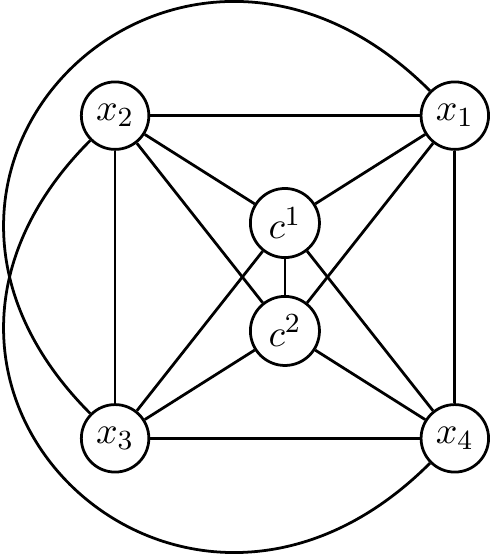}\label{fig:clause-gadget}}

\centering
  \subfloat[Bridge]%
  {\includegraphics[scale=.7]{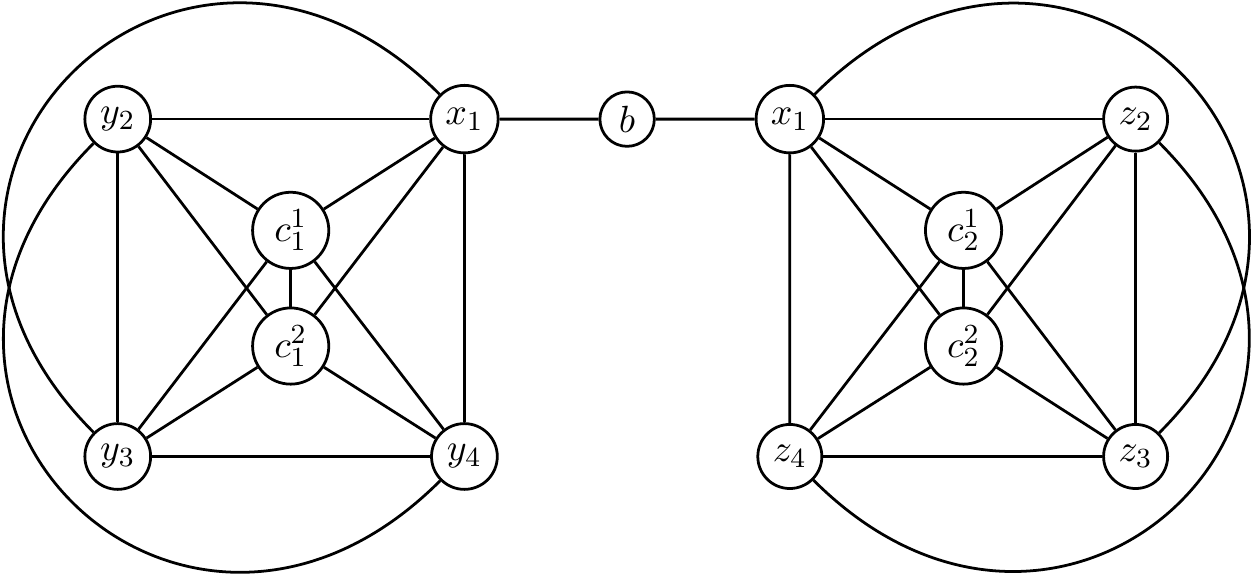}\label{fig:bridge-gadet}}
  \caption[Gadgets]{Clause gadget~\subref{fig:clause-gadget} for a
    clause $c = \nae_4(x_1, x_2, x_3, x_4)$ and a
    bridge~\subref{fig:bridge-gadet} connecting two vertices labeled
    by the same variable~$x_1$ in two different clauses
    $c_1 = \nae_4(x_1,y_2, y_3, y_4)$ and
    $c_2 = \nae_4(x_1,z_2, z_3, z_4)$.}
  \label{fig:gadgets}
\end{figure}

To prove 
Theorem~\ref{th:mcc_2} 
we use a result due to Creignou and
Hermann~\cite{ar:CreignouHermann96}. Let $\nae_k$ be a Boolean relation of
arity~$k$ of all tuples having at least one~$0$ and at least one~$1$,
i.e. $\nae_k = \{0,1\}^k \smallsetminus \{0\cdots 0, 1 \cdots1\}$,
commonly called the not-all-equal relation. Let
$\nae_k(x_1, \dots, x_k)$ be a constraint which is satisfied only by
all tuples from the relation $\nae_k$. A $\nae_k$ formula
$\varphi = c_1 \land \cdots \land c_n$ is satisfied if each
clause~$c_i$ is not-all-equal satisfied. Given a CNF
formula~$\varphi$, the counting problem $\#\nae_k\sat$ asks for the
number of not-all-equal satisfying assignments of~$\varphi$.

\begin{theorem}[\cite{ar:CreignouHermann96}]
  $\#\nae_k\sat$ is $\sharpP$-complete for each $k \geq 3$.
\end{theorem}

\begin{proof}[Proof of Theorem \ref{th:mcc_2}]
  The membership in $\sharpP$ is clear, therefore we focus on the
  proof of $\sharpP$-hardness. We perform a parsimonious reduction
  from $\#\nae_{t+1}\sat$.

  Given a $\nae_{t+1}$ formula~$\varphi$, we associate with it the
  graph $G_\varphi$ in the following way. For each clause
  $c = \nae_{t+1}(x_1, \ldots, x_{t+1})$ construct the clause gadget
  (as in Figure~\ref{fig:clause-gadget} for $t=3$), consisting of a
  complete graph~$K_{2t}$. Label $t+1$ vertices of~$K_{2t}$ by the
  variables~$x_1$, \ldots, $x_{t+1}$ and the remaining $t-1$ vertices
  by~$c^1$, \ldots, $c^{t-1}$, where~$c$ is the identifier of the
  clause.

  Connect different copies of the same variable in two clause gadgets
  by a bridge with a new vertex~$b$. There is a new bridge vertex~$b$
  for each pair of variable vertices labeled by the same variable~$x$
  in different clause gadgets.  Figure~\ref{fig:bridge-gadet}
  illustrates (for $t=3$) how two copies of the variable~$x_1$ in
  clauses $c_1 = \nae_{t+1}(x_1, y_2, \ldots, y_{t+1})$ and
  $c_2 = \nae_{t+1}(x_1, z_2, \ldots, z_{t+1})$ are connected through
  a bridge with a new node~$b$.

  We show that each $mcc_t$-coloring of the graph~$G_\varphi$ encodes
  a $\nae_{t+1}$-satisfiability of the formula~$\varphi$. A valid
  $mcc_t$ $2$-coloring of each clause gadget forces $t$~variables to
  be colored by the color~$0$ and the $t$~others by the
  color~$1$. Restricted to the variable nodes~$x_1$, \ldots, $x_{t+1}$
  of the clause gadget, this represents a correct assignment of the
  constraint $\nae_{t+1}(x_1, \ldots, x_{t+1})$. No other $2$-coloring
  of the clause gadget is a valid $mcc_t$ coloring.

  When two occurrences of the same variable~$x$ are connected through
  a bridge with a vertex~$b$, this bridge vertex~$b$ must be colored
  by a different color than the vertices labeled by~$x$. Indeed, in a
  valid $mcc_t$ coloring, the vertex~$x$ in the clause gadget is
  connected to another $t-1$ vertices colored by the same color. This
  induces a different color for the vertex~$b$, otherwise there would
  be a connected component containing $t+1$ vertices and colored by
  the same color. This also forces the two copies of the variable~$x$,
  connected by a bridge, to be colored by the same color.

  Hence, the satisfying assignments of a $\nae_{t+1}\sat$
  formula~$\varphi$ are in one-to-one correspondence to the $mcc_t$
  $2$-colorings of~$G_\varphi$, which constitutes a parsimonious
  reduction from $\#\nae_{t+1}\sat$.
\end{proof}



Next we determine the complexity of $\chi_{mcc_t}(G; k)$ for $k,t \in \N$.
\begin{theorem}
\label{th:mcc-new}
For any integers $t$ and $k$ that are both at least two, $\chi_{mcc_t}(G; k)$ is $\#$\textbf{P}-complete.
\end{theorem}

\begin{proof}
Membership in $\sharp\bP$ is clear. We shall show that if $t$ and $k$ are integers and both at least two, then
$\chi_{mcc_t}(G; k)$ is polynomial time reducible to $\chi_{mcc_t}(G; k+1)$. 
Combining this with Theorem \ref{th:mcc_2}
gives the result.

Given a graph $G$, form $G'$ from $G$ by adding a clique on $t(k+1)$ new vertices and joining one of the new vertices to every vertex of $G$. In any valid coloring of $G'$ with $k+1$ colors, each color class must contain precisely $t$ of the new vertices and these vertices can be colored in $\binom{(k+1)t}{t,t,\ldots,t}$ ways, where $\binom{(k+1)t}{t,t,\ldots,t}$ is the multinomial coefficient counting the number of ways of choosing an ordered collection of $k+1$ subsets each of size $t$ from a set of size $(k+1)t$.
Once the new vertices have been colored there are $k$ colors available to color the vertices of $G$. Thus
\[ \chi_{mcc_t}(G'; k+1) = \binom{(k+1)t}{t,t,\ldots,t}\chi_{mcc_t}(G; k).\]
The result follows.
\end{proof}

\begin{remark}
The statement
\[ \chi_{mcc_t}(G'; k+1) = \binom{(k+1)t}{t,t,\ldots,t}\chi_{mcc_t}(G; k).\]
appears at first sight to be an instance of
Linial's Trick, but there is a subtle difference.
In the proof above the choice of $G'= f(G,t,k)$ both depends on $t$ and $k$.

In applying Linial's Trick we require
that the graph $G' =g(G)$ does not depend on $k$
in order to get a polynomial identity of the form
\[ \chi_{mcc_t}(g^n(G); k+n) = f(k,n,t) \cdot \chi_{mcc_t}(G; k).\]

Currently we do not know whether
there is a version of Linial's Trick which can be used
for $\chi_{mcc_t}(G; X)$.
\end{remark}

Using the fact that $\chi_{mcc_1}(G; X)$ is the chromatic polynomial and the previous discussion
(Proposition \ref{prop:mcc-basic}, and Theorems \ref{th:mcc_2} and \ref{th:mcc-new})
we get

\begin{corollary}
\label{cor:mcc}
Evaluating $\chi_{mcc_t}(G; X)$ is in $\bP$ for $t=1$ and $k=0,1,2$, and for $t \geq 2$ and $k=1$.
For all other values of $t,k \in \N$ evaluation is $\sharp\bP$-complete.
\end{corollary}
\begin{problem}

\label{prob-8}
What is the complexity spectrum for $\chi_{mcc_t}(G;a)$ for $t \geq 2$ and $a \in \F - \N$?
\end{problem}

In Theorem \ref{th:cliquealpha} we have determined completely the complexity spectrum for $\chi_{DU(K_{\alpha})}(G;X)$.
This was meant to be a warm-up exercise for determining the complexity spectrum for $\chi_{mcc_t}(G; X)$,
as each $DU(K_{\alpha})$-coloring is also an $mcc_{\alpha}$-coloring. 
However, 
determining the difficulty of evaluating $\chi_{mcc_t}(G; X)$ seems to be much more demanding.

\subsection{$H$-free-colorings}
\label{se:hfree}

A function $f:V(G)\to[k]$ is an $H$-free coloring
if no color class induces a graph isomorphic to $H$. 
Clearly this is a $\Zilber$-property, hence
$\chi_{H-free}(G;k)$ is a polynomial in $k$.
The discrete complexity spectrum is rather well understood:

\begin{theorem}
\label{th:h-free}
\begin{enumerate}[(i)]
\item
(\cite{achlioptas1998existence})
$\chi_{H-free}(G;k)$
is $\sharp\bP$-hard for every $k\geq3$ and $H$ with at least $2$ vertices. 
\item
(\cite{ar:Achlioptas1997}) 
$\chi_{H-free}(G;2)$
is $\bNP$-hard for every $H$ with at most $2$ vertices. 
\end{enumerate}
\end{theorem}

It is easy to see that for $X=0,1$ 
we have the following evaluations:
$$\chi_{H-free}(G;0) = 0$$ and
$$\chi_{H-free}(G;1) =  \begin{cases}
1 & \mbox{if   } G \mbox{ is  } H-\mbox{free} \\
0 & \mbox{ otherwise }.
\end{cases}
$$

\begin{problem}
\label{prob-9}
\label{p:hfree}
What is the complexity of evaluation of $\chi_{H-free}(G;a)$ for $a \in \F - \N$?
\end{problem}

$H$-free coloring is another case
where Linial's Trick does not seem to work.

\section{More graph polynomials}
\label{se:varia}

In this section we discuss some graph polynomials
for which we have only partial knowledge, if any, about the complexity spectrum.
The graph polynomials $\chi_\Phi(G;k)$ we discuss arise from $\Zilber$-properties $\Phi$
defined in Section~\ref{subse:more}. 
Three of them, (i-iii), belong to the framework of $\mathcal{P}_1-\mathcal{P}_2$-colorings
from Subsection \ref{subse:P-colorings}, listed in Table~\ref{table-zilber},
and the remaining two, (iv,v), are mentioned here because they have a rich literature.

In each case we do not know --- but suspect --- that evaluation  of $\chi_\Phi(G;X)$ is $\sharp\bP$-hard
for $X=a$ for at least one $a \in \F$. 
\begin{problem}
\label{prob-10}
Determine the full complexity spectrum of $\chi_\Phi(G;X)$ for each of the graph polynomials in Section~\ref{subse:more}.
\end{problem}

Instead of counting colorings one can look at the corresponding decision problem which asks whether $\chi_\Phi(G;k) > 0$.
Clearly, the counting problem is at least as hard as this decision problem.
For each of the graph properties $\Phi$ above we do know that computing the polynomial $\chi_\Phi(G;X)$ is $\bNP$-hard. 
Furthermore, in two cases we discuss dichotomy theorems showing that 
an evaluation of $\chi_{\Phi}(G;X)$ is either in $\bP$ or $\bNP$-hard.

\subsection{Graph polynomials with incomplete complexity spectrum}
\label{subse:more}
We consider the following colorings:
\begin{enumerate}[(i)]
\item
Let $t\in\mathbb{N}$. A function $c:V(G)\rightarrow [k]$
is a {\em $t$-improper coloring} if every color induces a graph of maximal
degree $t$. 
$t$-improper colorings were studied in \cite{ar:CowenGoddardJesurum1997}.
They originate in certain network problems.

\item 
A function $c:V(G)\rightarrow [k]$ is an {\em acyclic coloring}
if it is proper and there is no two colored cycle in $G$. 
Acyclic colorings were introduced in \cite{ar:Grunbaum1973} and further studied in \cite{ar:AlonMcDiarmidReed1991}.

\item 
A function $f:V(G)\rightarrow [k]$ is a {\em co-coloring} if
every color class induces a graph which is either a clique or an independent
set. Co-colorings were first studied in \cite{ar:LesniakStraight77}. 
\item 
A function $f:E(G)\rightarrow [k]$ is a rainbow-path coloring (rainbow coloring for short)
if every two vertices are connected by a path in which every two edges
are colored differently. Rainbow colorings were first introduced in \cite{ar:Chartrand2008}. 
\item 
A function $f:V(G)\to[k]$ is an injective coloring
if it is injective on the open neighborhood of every vertex. 
$f$ does not have to be a proper coloring.
In other words,
if there is a path of length $2$ between $v$ and $u$ then $u$ and $v$ must
have different colors. 
\end{enumerate}
Let $\Phi_{t-imp}$, $\Phi_{acyc}$, $\Phi_{co-co}$, $\Phi_{rainbow}$, and $\Phi_{inject}$
denote respectively the coloring properties of $t$-improper colorings for fixed $t \in \mathbb{N}$, 
acyclic colorings, co-colorings, rainbow colorings, and injective colorings. 
Clearly, each of these coloring properties $\Phi$ is a $\Zilber$-property, and hence
the number of colorings in each $\Phi$ is a polynomial in $k$. 
We denote by $\chi_{t-imp}(G;k)$, $\chi_{acyc}(G;k)$, $\chi_{co-co}(G;k)$, $\chi_{rainbow}(G;k)$, and $\chi_{inject}(G;k)$
the graph polynomials $\chi_{\Phi}(G;k)$ counting colorings with at most $k$ colors of a graph $G$ in the corresponding $\Phi$.

For acyclic colorings, co-colorings, rainbow colorings and injective colorings, 
the complexity spectrum is completely unknown. 
For $t$-improper colorings, partial results are known. 
For $t=0$, the $t$-improper colorings are exactly the proper colorings, hence
$\chi_{0-imp}(G;k)= \chi(G;k)$
and the complexity spectrum is completely understood.
For $t=1$ and $k \in \N$ we have 
$\chi_{1-imp}(G;k)= \chi_{mcc_2}(G;k)$
and the complexity spectrum is completely understood.
For $t=2$ every color class consists of a disjoint union of 
paths and cycles.
This is the first case where the complexity spectrum of $\chi_{2-imp}(G;k)$ 
is not known.

\subsection{$\bNP$-hardness}

From the literature, we have that each of the graph polynomials defined in Section~\ref{subse:more}
is $\bNP$-hard: 
\begin{theorem}
\label{th:more-NP-hard}
\begin{enumerate}[(i)]
 \item 
Let $t\in \mathbb{N}$. 
Let $\Phi$ be one of $\Phi_{t-imp}$ $\Phi_{co-co}$, $\Phi_{rainbow}$, and $\Phi_{inject}$. 
Computing the minimal $k \in \N$  such that $\chi_{\Phi}(G;k)>0$ is $\bNP$-hard. 
\item 
It is NP-hard to
decide for a given $G$ and $k$ if the acyclic chromatic number of
$G$ is at most $k$. 
\end{enumerate}
\end{theorem}
\begin{proof}
The case of $t$-improper colorings in Theorem~\ref{th:more-NP-hard}(i) follows directly from \cite{ar:CowenGoddardJesurum1997}. 
The case of co-colorings is proven in~\cite{ar:GimbelKratschStewart1994}.
The case of rainbow colorings is proven in~\cite{ar:ChakrabortyFischerMatsliahYuster2008}. 
The case of injective colorings follows directly from~\cite{ar:HahnKratochvilSiranSotteau2002}. 
Theorem~\ref{th:more-NP-hard}(ii) is proven in~\cite{phd:Kostochka}. 
\end{proof}

For $t$-improper colorings we are able to give a dichotomy theorem for graphs with multiple edges:

\ifrevised
\begin{proposition}
\label{prop:t-imp}
\begin{enumerate}[(i)]
\item
For every $a \in \F - \{0,1\}$ the problem of evaluating $\chi_{t-imp}(G;a)$,
where the input runs over the class of graphs with multiple edges allowed, is $\bNP$-hard.
\item
The evaluations $\chi_{t-imp}(G;0)$ and $\chi_{t-imp}(G;1)$ are in $\bP$.
\end{enumerate}
\end{proposition}
\fi 
\begin{proof}
Let $G\bowtie_t K_{1}$ be the graph obtained from $G$ by adding
a new vertex $v$ and putting $t+1$ edges between $v$ and any vertex
of $G$. Clearly, $v$ cannot be colored the same color as any other
vertex of $G$. Therefore, we can use Linial's Trick. 
\end{proof}

For acyclic colorings, we have a partial dichotomy.
The evaluations of $\chi_{acyc}(G;X)$ with $X=0$ and $X=1$ are trivial by definition, because every proper coloring
with less than two colors is acyclic.

\begin{proposition}
\label{pr:acyc}
For every $a \notin\mathbb{N}$, it is NP-hard to compute $\chi_{acyc}(G; a)$. 
\end{proposition}
\begin{proof}
We use a version of Linial's Trick for the chromatic polynomial.
Consider $G\bowtie K_{1}$. Every acyclic coloring $f$ of $G\bowtie K_{1}$
with color set $[k]$ must color the vertex $K_{1}$ with a unique
color. The coloring $f\mid_{G}$ induced by $f$ on $G$ is clearly
proper and does not contain any two-colored cycles, since $f$ is
acyclic. 
On the other hand, every acyclic coloring $g$  of $G$ with color
set $[k-1]$ can be transformed to an acyclic coloring  $f$ of $G\bowtie K_{1}$
with color set $[k]$ by coloring $K_{1}$ with any of the $k$ colors,
and then setting 
\[
f(v)=\begin{cases}
g(v) & g(v)<f(K_{1})\\
g(v)+1 & g(v)\geq f(K_{1})
\end{cases}
\]
 for any $v\in V(G)$. Hence we get 
\[
\chi_{acyc}(G\bowtie K_{1};x_{0})=\chi_{acyc}(G;x_{0}-1)\cdot x_{0}\,.
\]
Using Theorem \ref{th:more-NP-hard}(ii) we complete the proof.
\end{proof}
A solution of the following problem would complete the dichotomy:
\begin{problem}
\label{prob-11}
What is the complexity of the evaluation $\chi_{acyc}(G;a)$ for $a \in \N -\{0,1\}$?
\end{problem}

\ifskip
\else
\begin{framed}
The following table needs completions, if possible.
\end{framed}

\subsection*{Summary}

\small
\begin{center}
\begin{table}[htb]
\label{table-varia}
\begin{tabular}{|l | l| l| l | l|}
\hline
&&&& \\
G-polynomial & $E=\mathrm{EASY}(P)$ & $\mathrm{\bNP HARD}(P)$ & $\mathrm{OTHER}$ & Reference\\
&&&&\\
\hline
\vspace{0.05cm}
$\chi_{t-imp}(G;X)$ & $E_{t-imp}= \{0,1\}$ & $\F- E_{t-imp}$ & $\emptyset$ & Theorem \ref{prop:t-imp}\\
\hline
\vspace{0.05cm}
$\chi_{acyc}(G;X)$ & $ \{0,1\} \subseteq E_{acyc}$ & contains $\F- E_{acyc}$ & ??? & Theorem \ref{pr:acyc}\\
\hline
\vspace{0.05cm}
$\chi_{co-co}(G;X)$ & ???  & ??? & ??? & ???\\
\hline
\vspace{0.05cm}
$\chi_{rainbow}(G;X)$ & ???  & ??? & ??? & ???\\
\hline
\vspace{0.05cm}
$\chi_{injec}(G;X)$ & ???  & ??? & ??? & ???\\
\hline
\end{tabular}
\caption{Incomplete complexity spectra}
\label{table-incomplete}
\end{table}
\end{center}
\fi

\section{Conclusions and open problems}
\label{se:conclu}
In the light of the discovery of many univariate graph polynomials
in \cite{ar:KotekMakowskyZilber11}, and the fact that infinitely many
of them
are incomparable in expressive power, Theorem \ref{th:dp},
we initiated the systematic study of 
these graph polynomials.
Inspired by N. Linial's work in \cite{ar:Linial86} we have concentrated
in this paper on
the complexity of evaluating these graph polynomials.
We have introduced the full and the discrete complexity spectrum of 
univariate graph polynomials.
In this paper
we concentrated our attention on graph polynomials arising from graph colorings
previously studied in the literature.

Throughout the paper we have listed\footnote{
These are Problems 
\ref{prob-1},
\ref{prob-3},
\ref{prob-4},
\ref{prob-5},
\ref{prob-6},
\ref{prob-8},
\ref{prob-9},
\ref{prob-10} and
\ref{prob-11}.
} open problems we encountered in our
explorations.  They mostly ask for the complete determination of the full complexity
spectrum for specific graph polynomials.
Among the more interesting challenges we have the following:

\begin{problem}
Determine the  full complexity spectrum of
\begin{enumerate}[(i)]
\item
the edge chromatic polynomial
$\chi_{edge}(G;X)$; 
\item
the polynomial
$\chi_{mcc_t}(G;X)$ for $t\geq 2$;
\item
the polynomials
$\chi_{H-free}(G;X)$ for $t\geq 2$;
\item
the generalized chromatic polynomials 
derived from $t$-improper colorings, acyclic colorings, co-colorings, rainbow-path colorings
and injective colorings
of Section \ref{se:varia}.
\end{enumerate}
\end{problem}

We have not found a single graph polynomial which does not
have a complexity spectrum satisfying the Difficult Point Dichotomy.
Our results so far suggest that there might be  {\em meta-theorem} to be
formulated, and finally also to be proven, which says that for a large
class of graph polynomials Difficult Point Dichotomy holds.
The large class in question should be defined by some definability criterion.
Examples of such criteria could come from descriptive complexity theory.

Here are some candidates for logically defined classes\footnote{
For a discussion of logically defined classes of graph polynomials, \\ see
\cite{ar:MakowskyZoo,pr:MakowskyKotekRavve2013,ar:KotekMakowskyZilber11,phd:Kotek}.
}
 of univariate graph polynomials
\begin{enumerate}[(i)]
\item
All $\SOL$-definable graph polynomials.
\item
All $\MSOL$-definable graph polynomials.
\item
All $\mathcal{P}$-chromatic polynomials where $\mathcal{P}$
is in $\bNP$, or equivalently, where $\mathcal{P}$ is definable in
$\exists\SOL$, the existential fragment of $\SOL$.
\end{enumerate}

Previous experience suggests that it is too early to formulate a solid conjecture.
\begin{problem}
Formulate and prove a meta-theorem for the Difficult Point Dichotomy.
\end{problem}

We hope the search for a meta-theorem will spawn further research
and will lead to new insights both in graph theory and in descriptive complexity theory.

\subsection*{Acknowledgements}
We would like to thank two anonymous referees of an earlier version of this paper
for their accurate reading  and spotting gaps and imprecisions, and for their extremely valuable comments.
We would also like to thank the editors of this special issue, J. Ellis-Monaghan, J. Kung and I. Moffatt
for their patience and support.

\section*{References}

\end{document}